\documentclass[12pt,a4paper]{article}
\usepackage{amsfonts}
\usepackage{amsfonts}
\usepackage{amsfonts}
\usepackage{amsfonts}
\usepackage{amsfonts}
\usepackage{mathrsfs}
\usepackage{amsfonts}
\usepackage{mathrsfs}
\usepackage{amsfonts}
\usepackage{amsfonts}
\usepackage{mathrsfs}
\usepackage{mathrsfs}
\usepackage{amsfonts}
\usepackage{amsfonts}
\usepackage{amsfonts}
\usepackage{amsfonts}
\usepackage{mathrsfs}

\usepackage{amsfonts}
\usepackage{mathrsfs}
\usepackage{mathrsfs}
\usepackage{mathrsfs}
\usepackage{mathrsfs}
\usepackage{amsfonts}
\usepackage{mathrsfs}
\usepackage{mathrsfs}
\usepackage{amsfonts}
\usepackage{amsfonts}
\usepackage{amsfonts}
\usepackage{amsfonts}
\usepackage{amsfonts}
\usepackage{amsfonts}
\usepackage{amsfonts}
\usepackage{amsfonts}
\usepackage{amsfonts}
\usepackage{mathrsfs}
\usepackage{amsfonts}
\usepackage{mathrsfs}
\usepackage{amsfonts}
\usepackage{mathrsfs}
\usepackage{amsfonts}
\usepackage{amsfonts}
\usepackage{amsfonts}
\usepackage{amsfonts}
\usepackage{amsfonts}
\usepackage{amsfonts}
\usepackage{amsfonts}
\usepackage{amsfonts}
\usepackage{amsfonts}
\usepackage{amsfonts}
\usepackage{amsfonts}
\usepackage{amsfonts}
\usepackage{mathrsfs}
\usepackage{mathrsfs}
\usepackage{mathrsfs}
\usepackage{mathrsfs}
\usepackage{mathrsfs}
\usepackage{mathrsfs}
\usepackage{mathrsfs}
\usepackage{mathrsfs}
\usepackage{mathrsfs}
\usepackage{amsfonts}
\usepackage{amsfonts}
\usepackage{amsfonts}
\usepackage{amsfonts}
\usepackage{amsfonts}
\usepackage{amsfonts}
\usepackage{mathrsfs}
\usepackage{amsfonts}
\usepackage{amsfonts}
\usepackage{amsfonts}
\usepackage{amsfonts}
\usepackage{amsfonts}
\usepackage{amsfonts}
\usepackage{amsfonts}
\usepackage{amsfonts}
\usepackage{amsfonts}
\usepackage{amsfonts}
\usepackage{amsfonts}
\usepackage{amsfonts}
\usepackage{mathrsfs}
\usepackage{amsfonts}
\usepackage{mathrsfs}
\usepackage{amsfonts}
\usepackage{amsfonts}
\usepackage{amsfonts}
\usepackage{mathrsfs}
\usepackage{amsfonts}
\usepackage{amsfonts}
\usepackage{amsfonts}
\usepackage{amsfonts}
\usepackage{amsmath}
\usepackage{mathrsfs}
\usepackage{amsfonts}
\usepackage{amsfonts}
\usepackage{amsfonts}
\usepackage{amsfonts}
\usepackage{amsfonts}
\usepackage{amsfonts}
\usepackage{amsfonts}
\usepackage{amsfonts}
\usepackage{amsfonts}
\usepackage{amsfonts}
\usepackage{amsfonts}
\usepackage{mathrsfs}
\usepackage{amsfonts}
\usepackage{amsfonts}
\usepackage{amsfonts}
\usepackage{amsfonts}
\usepackage{amsmath}
\usepackage{mathrsfs}
\usepackage{mathrsfs}
\usepackage{mathrsfs}
\usepackage{mathrsfs}
\usepackage{mathrsfs}
\usepackage{mathrsfs}
\usepackage{mathrsfs}
\usepackage{mathrsfs}
\usepackage{amsfonts}
\usepackage{amsfonts}
\usepackage{amsfonts}
\usepackage{amsfonts}
\usepackage{amsfonts}
\usepackage{amsfonts}
\usepackage{amsfonts}
\usepackage{amsfonts}
\usepackage{amsfonts}
\usepackage{amsfonts}
\usepackage{amsfonts}
\usepackage{amsfonts}
\usepackage{mathrsfs}
\usepackage{amsfonts}
\usepackage{mathrsfs}
\usepackage{amsfonts}
\usepackage{amsfonts}
\usepackage{amsfonts}
\usepackage{amsfonts}
\usepackage{amsfonts}
\usepackage{amsfonts}
\usepackage{amsfonts}
\usepackage{amsfonts}
\usepackage{amsfonts}
\usepackage{amsfonts}
\usepackage{amsfonts}
\usepackage{amsfonts}
\usepackage{amsfonts}
\usepackage{amsfonts}
\usepackage{amsfonts}
\usepackage{amsfonts}
\usepackage{amsfonts}
\usepackage{amsfonts}
\usepackage{mathrsfs}
\usepackage{mathrsfs}
\usepackage{mathrsfs}
\usepackage{mathrsfs}
\usepackage{amsfonts}
\usepackage{amsfonts}
\usepackage{amsfonts}
\usepackage{amsfonts}
\usepackage{amsfonts}
\usepackage{amsfonts}
\usepackage{amsfonts}
\usepackage{amsfonts}
\usepackage{amsfonts}
\usepackage{amsfonts}
\usepackage{amsfonts}
\usepackage{amsfonts}
\usepackage{amsfonts}
\usepackage{amsfonts}
\usepackage{amsfonts}
\usepackage{amsfonts}
\usepackage{amsfonts}
\usepackage{mathrsfs}
\usepackage{mathrsfs}
\usepackage{amsfonts}
\usepackage{amsfonts}
\usepackage{amsfonts}
\usepackage{amsfonts}
\usepackage{amsfonts}
\usepackage{amsfonts}
\usepackage{amsfonts}
\usepackage{amsfonts}
\usepackage{amsfonts}
\usepackage{amsfonts}
\usepackage{amsfonts}
\usepackage{amsfonts}
\usepackage{amsfonts}
\usepackage{mathrsfs}
\usepackage{mathrsfs}
\usepackage{amsfonts}
\usepackage{amsfonts}
\usepackage{amsfonts}
\usepackage{amsfonts}
\usepackage{amsfonts}
\usepackage{amsfonts}
\usepackage{mathrsfs}
\usepackage{mathrsfs}
\usepackage{amsfonts}
\usepackage{amsfonts}
\usepackage{amsfonts}
\usepackage{amsfonts}
\usepackage{amsfonts}
\usepackage{amsfonts}
\usepackage{amsfonts}
\usepackage{amsfonts}
\usepackage{amsfonts}
\usepackage{amsfonts}
\usepackage{amsfonts}
\usepackage{amsfonts}
\usepackage{amsfonts}
\usepackage{amsfonts}
\usepackage{mathrsfs}
\usepackage{amsfonts}
\usepackage{amsfonts}
\usepackage{amsfonts}
\usepackage{amsfonts}
\usepackage{amsfonts}
\usepackage{amsfonts}
\usepackage{amsfonts}
\usepackage{amsfonts}
\usepackage{amsfonts}
\usepackage{amsfonts}
\usepackage{amsfonts}
\usepackage{amsfonts}
\usepackage{amsfonts}
\usepackage{amsfonts}
\usepackage{amsfonts}
\usepackage{amsfonts}
\usepackage{amsmath}
\usepackage{mathrsfs}
\usepackage{amsfonts}
\usepackage{amsfonts}
\usepackage{amsfonts}
\usepackage{amsfonts}
\usepackage{amsmath}
\usepackage{mathrsfs}
\usepackage{amsfonts}
\usepackage{mathrsfs}
\usepackage{mathrsfs}
\usepackage{mathrsfs}
\usepackage{mathrsfs}
\usepackage{amsmath}
\usepackage{mathrsfs}
\usepackage{amsfonts}
\usepackage{color}
\usepackage{mathrsfs}

\usepackage{stmaryrd}
\usepackage{amsmath}
\usepackage{amssymb}

\usepackage{bbm}
\usepackage{mathrsfs}
\usepackage{graphicx}

\usepackage{enumerate}
\usepackage{theorem}

\makeatletter
\def\tank#1{\protected@xdef\@thanks{\@thanks
        \protect\footnotetext[0]{#1}}}
\def\bigfoot{

    \@footnotetext}
\makeatother

\newcommand{\ea}{\end{array}}
\newtheorem{theorem}{Theorem}[section]
\newtheorem{proposition}{Proposition}[section]
\newtheorem{claim}{Claim}[section]

\newtheorem{lemma}{Lemma}[section]
\newtheorem{definition}{Definition}[section]
\newtheorem{remark}{Remark}[section]

{\theorembodyfont{\rmfamily}
}

\newenvironment{proof}{Proof.}

\def \eqref#1{\hbox{(\ref{#1})}}

\allowdisplaybreaks

\begin{document}
\title{\Large \bf Stochastic generalized porous media equations driven by L\'{e}vy noise with increasing Lipschitz nonlinearities\thanks{Weina Wu's research is supported by National Natural Science Foundation of China (NSFC) (No. 11901285, 11771187), China Scholarship Council (CSC) (No. 202008320239), National Statistical Science Research Project of China (No. 2018LY28), School Start-up Fund of Nanjing University of Finance and Economics (NUFE), Support Programme for Young Scholars of NUFE.
Jianliang Zhai's research is supported by NSFC (No. 11971456, 11671372, 11721101), School Start-up Fund (USTC) KY0010000036, the Fundamental Research Funds for the Central Universities (No. WK3470000016).} }

\author{{Weina Wu}$^a$\footnote{E-mail:wuweinaforever@163.com}~~~ {Jianliang Zhai}$^b$\footnote{E-mail:zhaijl@ustc.edu.cn}
\\
 \small  a. School of Economics, Nanjing University of Finance and Economics, Nanjing, Jiangsu 210023, China.\\
 \small  b.  CAS Wu Wen-Tsun Key Laboratory of Mathematics, School of Mathematical Sciences,\\
 \small University of Science and Technology of China, Hefei, Anhui 230026, China.}\,
\date{}
\maketitle

\begin{center}
\begin{minipage}{130mm}
{\bf Abstract.}
We establish the existence and uniqueness of strong solutions to stochastic porous media equations driven by L\'{e}vy noise on a $\sigma$-finite measure space $(E,\mathcal{B}(E),\mu)$, and with the Laplacian replaced by a negative definite self-adjoint operator. The coefficient $\Psi$ is only assumed to satisfy the increasing Lipschitz nonlinearity assumption, without the restriction $r\Psi(r)\rightarrow\infty$ as $r\rightarrow\infty$ for $L^2(\mu)$-initial data. We also extend the state space, which avoids the transience assumption on $L$ or the boundedness of $L^{-1}$ in $L^{r+1}(E,\mathcal{B}(E),\mu)$ for some $r\geq1$. Examples of the negative definite self-adjoint operators include fractional powers of the Laplacian, i.e. $L=-(-\Delta)^\alpha,\ \alpha\in(0,1]$, generalized Schr\"{o}dinger operators, i.e. $L=\Delta+2\frac{\nabla \rho}{\rho}\cdot\nabla$, and Laplacians on fractals.

\vspace{3mm} {\bf Keywords.} Stochastic porous media equation; L\'{e}vy noise; Sub-Markovian; Strongly continuous contraction semigroup.

\end{minipage}
\end{center}

\section{Introduction}
\setcounter{equation}{0}
 \setcounter{definition}{0}
Fix $T>0$. Let $(E,\mathcal{B})$ be a standard measurable space (see \cite{P67}) with a $\sigma$-finite measure $\mu$. Let $(\Omega,\mathcal{F},\mathbb{F},\Bbb{P})$, where $\mathbb{F}:=\{\mathcal{F}_t\}_{t\in[0,T]}$, be a complete filtered probability space satisfying the usual condition. We shall denote by $\mathcal{B}\mathcal{F}$ the $\sigma$-field of the progressively measurable sets on $[0,T]\times\Omega$, i.e.,
\begin{eqnarray*}
\mathcal{B}\mathcal{F}=\{A\subset[0,T]\times\Omega:\ \forall t\in[0,T],\ A\cap([0,t]\times\Omega)\in\mathcal{B}([0,t])\otimes\mathcal{F}_t\},
\end{eqnarray*}
where $\mathcal{B}([0,t])$ denotes the Borel $\sigma$-field on $[0,t]$. Let $(Z, \mathcal{Z})$ be a measurable space, and $\nu$ be a $\sigma$-finite positive measure on it. We assume that $N$ is a time-homogeneous Poisson random measure on $[0,T]\times Z$ with the intensity measure $\lambda_T\otimes\nu$ on $(\Omega,\mathcal{F},\mathbb{F},\Bbb{P})$, where $\lambda_T$ is the Lebesgue measure on $[0,T]$. We define the compensated Poisson random measure $\widetilde{N}$ by
\begin{eqnarray*}
\widetilde{N}((0,t]\times B)=N((0,t]\times B)-t\nu(B),\ \forall t\in[0,T],\ B\in\mathcal{Z}\text{ with }\nu(B)<\infty.
\end{eqnarray*}

The purpose of this paper is to establish the existence and uniqueness of strong solutions to the following stochastic generalized porous media equations driven by L\'{e}vy process:
\begin{equation} \label{eq:1}
\left\{ \begin{aligned}
&dX(t)=L\Psi(X(t))dt+\int_{Z}f(t,X(t-),z)\widetilde{N}(dt,dz),\ \text{in}\ [0,T]\times E,\\
&X(0)=x \text{~on~} E,
\end{aligned} \right.
\end{equation}
where $L$ is the infinitesimal generator of a symmetric sub-Markovian strongly continuous contraction semigroup $(P_t)_{t\geq0}$ on $L^2(\mu):=L^2(E,\mathcal{B}(E),\mu)$.
$\Psi(\cdot):\Bbb{R}\rightarrow\Bbb{R}$ is a monotonically nondecreasing Lipschitz continuous function. $f:[0,T]\times \Omega\times F^*_{1,2}\times Z\rightarrow F^*_{1,2}$ is a $\mathcal{B}\mathcal{F}\otimes\mathcal{B}(F^*_{1,2})\otimes\mathcal{Z}$-measurable function. For the definition of the Hilbert space
$F^*_{1,2}$ and the precise conditions on $\Psi$ and $f$, we refer to Section 2 and Section 3 respectively.

The study of porous media equations has attained much interest in recent years. Recall that the classical porous media equation reads
$$dX(t)=\Delta X^m(t)dt$$
on a domain in $\Bbb{R}^d$, we refer to \cite{A} for its mathematical treatment and the physical background, and also \cite[Section 4.3]{B} for the general theory of equations of such type. Stochastic porous media equations (SPMEs) have been intensively discussed since the foundational work in \cite{DR, DR04}. Meanwhile, there are plenty results about the existence and uniqueness of solutions and their long-time behaviors of SPMEs driven by Wiener process on general measure spaces (\cite{BDPR04, DRRW06, RRW, RW, RW07, RWW, RWX, RWX1, RWZ, W, WZ}). However, to the best of our knowledge, there seem to be very few results about SPMEs driven by L\'{e}vy-type or Poission-type perturbations on general measure spaces. Hou and Zhou investigated the existence and uniqueness of solutions to SPMEs driven by L\'{e}vy noise on a separable probability space in \cite{ZH} in a variational setting, following the approach of \cite{RRW}. Based on the methods used in \cite{ZH}, the ergodicity and the exponential stability of the same equation were obtained in \cite{ZH1} and \cite{GZ} respectively. In our framework, we consider \eqref{eq:1} in $\sigma$-finite measurable spaces. We would also like to emphasize that the state space we work in is $F^*_{1,2}$, which is larger than the dual space of the extended Dirichlet space considered in \cite{RRW, ZH, ZH1, GZ}, hence we can allow more general initial conditions. In our case, we do not need the transience assumption of the Dirichlet form as in \cite{RRW} or the boundedness of $L^{-1}$ in $L^{r+1}(E,\mathcal{B}(E),\mu)$ for some $r\geq1$ as in \cite{ZH, ZH1, GZ}. In addition, in \cite{RRW, ZH, ZH1, GZ}, $\Psi$ is assumed to be continuous such that $r\Psi(r)\rightarrow\infty$ as $r\rightarrow\infty$. In this paper, we show that for Lipschitz continuous $\Psi$ this condition can be dropped for $L^2(\mu)$-initial data.

The main contribution of this work is that we obtain the existence and uniqueness of strong solutions to \eqref{eq:1} driven by L\'{e}vy noise, which generalizes many previous works \cite{BRR, RRW, RW, RWX, ZH}. This work is inspired by the recent paper \cite{RWX}, in which the first author together with R\"{o}ckner and Xie proved the existence and uniqueness of strong solutions to \eqref{eq:1} driven by Wiener process. In this paper, we will follow the approximating techniques in \cite{RWX} and the local monotonicity arguments (cf. \cite{BLZ}). However, since we consider \eqref{eq:1} driven by L\'{e}vy noise, our proofs are more involved with a substantial number of obstacles to overcome, which do not occur in \cite{RWX}.

We use the same assumption of $\Psi$ (see \textbf{(H1)} below in Section 3) as in \cite{BRR, RWX}, the obvious and useful consequences are \eqref{eqn38} and \eqref{eqn40}, which play crucial technical roles to prove the local monotonicity in Step 1 (see pg:7 and \cite[Step 1 (ii)]{RWX}), to obtain \eqref{eqn25}, \eqref{eqn381} and so on. Compared with \cite{RRW}, $\Psi(x)=x^3$ can not be covered in this paper or \cite{BRR, RWX}, but since we work in larger state space $F^*_{1,2}$ as in \cite{RWX}, the transience assumption of the Dirichlet form can be dropped, in particular, we do not need any restriction on $d$ when $E=\Bbb{R}^d$ and $L=(-\Delta)^\alpha, \alpha\in(0,1]$. The assumptions of $(E,\mathcal{B}(E),\mu)$ and $(L,D(L))$ are the same as in \cite{RWX}, and $(L,D(L))$ is indeed the associated Dirichlet operator on $L^2(\mu)$ (cf. \cite{MR}). There are plenty examples of $(E,\mathcal{B}(E),\mu)$ and $L$ that can not be covered by \cite{ZH, ZH1, GZ}, in which the authors considered the equations on separable probability spaces. For example, let $E:=U\subset \Bbb{R}^d$, $U$ open, and $m$ a positive Radon measure on $U$ such that $supp[m]=U$, using Dirichlet form theory, one can define its associated Dirichlet operator $L$ on $L^2(U,m)$. If $L$ is the Friedrichs extension of the operator $L_0=\Delta+2\frac{\nabla \rho}{\rho}\cdot\nabla$ on $L^2(\Bbb{R}^d, \rho^2dx)$, where $dx$ denotes Lebesgue measure and $\rho\in H^1(\Bbb{R}^d)$ ($H^1$ is the usual Sobelev space), one can prove there exists a Dirichlet form on $L^2(\Bbb{R}^d, \rho^2dx)$ and $L$ is the corresponding Dirichlet operator. Furthermore, if $E$ is a fractal, we can take $L$ to be the Laplace operator on this fractal. We refer to \cite[Section 4]{RWX} and references therein for more examples and explicit verifications of the above $E$ and $L$.

Finally, we would like to refer \cite{LR, LR1, P, PR} for more background information and results on SPDEs, \cite{A, BDR} on SPMEs, \cite{RRW, RW, RWW, RWX, RWX1} and the references therein for comprehensive theories of stochastic generalized porous media equations.

The paper is organized as follows: in Section 2, we introduce some notations and recall some known results for preparation. In Section 3, we present our assumptions and the main result. In Section 4, we construct approximated equations for $L$ and $\Psi$, and by using the local monotonicity arguments, we show the existence and uniqueness of solutions to the approximated equations. We also obtain some a priori estimates for those approximated solutions. Section 5 is devoted to prove that the limit of those approximated solutions solves \eqref{eq:1}.

\section{Notations and Preliminaries}
\setcounter{equation}{0}
 \setcounter{definition}{0}

Let us first recall some basic definitions and spaces which will be used throughout the paper (see \cite{RWX, WZ}).

Let $(E,\mathcal{B},\mu)$ be a $\sigma$-finite measurable space, which we fix in the entire paper. We assume that $(E, \mathcal{B})$ is a standard measurable space (i.e., $\sigma$-isomorphic to a Polish space, see \cite{P67}). Let $(P_t)_{t\geq0}$ be a strongly continuous, symmetric sub-Markovian
contraction semigroup on $L^2(\mu)$ with generator $(L, D(L))$. The $\Gamma$-transform of $(P_t)_{t\geq0}$ is defined by the following Bochner integral (\cite{HK})
\begin{eqnarray}\label{eqnarray1}
V_ru:=\frac{1}{\Gamma(\frac{r}{2})}\int_0^\infty t^{\frac{r}{2}-1}e^{-t}P_tudt,~~u\in L^2(\mu),~~r>0.
\end{eqnarray}
In this paper, we consider the Hilbert space $(F_{1,2},\|\cdot\|_{F_{1,2}})$ defined by
$$
F_{1,2}:=V_1(L^2(\mu)),~\text{with~norm}~\|f\|_{F_{1,2}}=|u|_2,~~\text{for}~~f=V_1u,~~ u\in L^2(\mu),
$$
where the norm $|\cdot|_2$ is defined as $|u|_2=(\int_E |u|^2d\mu)^{\frac{1}{2}}$, and its inner product is denoted by $\langle \cdot, \cdot\rangle_2$.
From \cite{FJS1,FJS2}, we know
$$
V_1=(1-L)^{-\frac{1}{2}},~~\text{so~that}~~F_{1,2}=D\big((1-L)^{\frac{1}{2}}\big)~~\text{and}~~\|f\|_{F_{1,2}}=|(1-L)^{\frac{1}{2}}f|_2.
$$
The dual space of $F_{1,2}$ is denoted by $F^*_{1,2}$ and $F^*_{1,2}=D((1-L)^{-\frac{1}{2}})$, it is equipped with norms
\begin{eqnarray}\label{eqnarray5}
\|\eta\|_{F^*_{1,2,\varepsilon}}:=\langle \eta, (\varepsilon-L)^{-1}\eta\rangle_2^{\frac{1}{2}},~~\eta\in F^*_{1,2},~~0<\varepsilon<\infty.
\end{eqnarray}
Denote the duality between $F^*_{1,2}$ and $F_{1,2}$ by
$_{F^*_{1,2}}\langle \cdot,\cdot\rangle_{F_{1,2}}$.

\vspace{2mm}
Consider the Dirichlet form $(\mathscr{E}, D(\mathscr{E}))$ on
$L^2(\mu)$ associated with $(L, D(L))$, i.e.,
\begin{eqnarray*}
   && D(\mathscr{E}):= F_{1,2},~~ \text{and} \\
   && \mathscr{E}(u,v):=\mu(\sqrt{-L}u \sqrt{-L}v):=\int_E\sqrt{-L}u \sqrt{-L}vd\mu,~~ u,v\in F_{1,2}.
\end{eqnarray*}
Let $D(\mathscr{E})$ be equipped with the inner product
$\mathscr{E}_1:=\mathscr{E}+\langle\cdot,\cdot\rangle_2$. Consider the inner product
$\mathscr{E}_\varepsilon:=\mathscr{E}+\varepsilon \langle\cdot, \cdot \rangle_2,~ \varepsilon\in(0,\infty)$, on
$F_{1,2}$, i.e.,
\begin{eqnarray*}\label{eqnal101}
  \|v\|^2_{F_{1,2,\varepsilon}}:=\mathscr{E}(v,v)+\varepsilon\int|v|^2d\mu,\ \text{for}\ v\in F_{1,2},
\end{eqnarray*}
and
\begin{eqnarray*}\label{eqnal202}
 \|\eta\|_{F^*_{1,2,\varepsilon}}:=_{F^*_{1,2}}\langle \eta, (\varepsilon-L)^{-1}\eta\rangle^{\frac{1}{2}}_{F_{1,2}}:=\sup_{\substack{v\in
F_{1,2}\\ \|v\|_{F_{1,2,\varepsilon}}\leq1}}\eta(v),~\eta\in F_{1,2}^*.
\end{eqnarray*}
The sequence of norms $\|\cdot\|_{F^*_{1,2,\varepsilon}}$, $0<\varepsilon<\infty$, are in fact equivalent. To keep the completeness of presentation, we cite its proof from a forthcoming paper \cite[Proposition 2.1]{RWX1}, which is written by the first author, R\"{o}ckner and Xie.
\begin{proposition}\label{pro1}
Let $\eta\in F^*_{1,2}$. Then $\varepsilon\mapsto\|\eta\|_{F^*_{1,2,\varepsilon}}$ is decreasing, and
\begin{eqnarray}
&&\|\eta\|_{F^*_{1,2}}\leq \|\eta\|_{F^*_{1,2,\varepsilon}}\leq
\frac{1}{\sqrt{\varepsilon}}\|\eta\|_{F^*_{1,2}},~~\forall~0<\varepsilon<1.\label{eqn22}
\end{eqnarray}
\end{proposition}
\begin{proof}
Note that for all $\eta\in F_{1,2}^*$ and $0<\varepsilon'\leq \varepsilon<\infty$, we have
\begin{eqnarray*}
\|\eta\|_{F^*_{1,2,\varepsilon}}=:
\sup_{\substack{{v\in
F_{1,2}}\\{\|v\|_{F_{1,2,\varepsilon}\leq1}}}}\eta(v)\leq \sup_{\substack{{v\in
F_{1,2}}\\{\|v\|_{F_{1,2,\varepsilon'}\leq1}}}}\eta(v)=\|\eta\|_{F^*_{1,2,\varepsilon'}},
\end{eqnarray*}
i.e., $\forall~\eta\in F^*_{1,2}$, $\|\eta\|_{F^*_{1,2,\varepsilon}}$ is decreasing in $\varepsilon$. In particular, the first inequality in \eqref{eqn22} holds. For $0<\varepsilon<1$, one has that
\begin{eqnarray*}
\|\eta\|_{F^*_{1,2}}=\sup_{\|v\|_{F_{1,2}}\leq1}\eta(v)
=\sqrt{\varepsilon}\sup_{\sqrt{\varepsilon}\|v\|_{F_{1,2}}\leq1}\eta(v)
\geq\sqrt{\varepsilon}\sup_{\|v\|_{F_{1,2,\nu}}\leq1}\eta(v)
=\sqrt{\varepsilon}\|\eta\|_{F^*_{1,2,\varepsilon}},
\end{eqnarray*}
which yields the second inequality in \eqref{eqn22}. Similarly, we can also obtain that
\begin{eqnarray}\label{equivalent norm 2}
&&\frac{1}{\sqrt{\varepsilon}}\|\eta\|_{F^*_{1,2}}\leq \|\eta\|_{F^*_{1,2,\varepsilon}}\leq
\|\eta\|_{F^*_{1,2}},~~\forall~1\leq\varepsilon<\infty.
\end{eqnarray}
\hspace{\fill}$\Box$
\end{proof}

\begin{remark}
This property will be used in the proofs of Claim \ref{claim2}, Claim \ref{claim3} and Proposition \ref{Th2}. By considering the proofs in $(F^*_{1,2},\|\cdot\|_{F^*_{1,2,\varepsilon}})$ instead of $(F^*_{1,2},\|\cdot\|_{F^*_{1,2}})$, we can avoid some technical obstacles. Taking Claim \ref{claim2} for example, if we apply It\^{o}'s formula to $\|X^\varepsilon_\lambda(t)-X^\varepsilon_{\lambda'}(t)\|^2_{F^*_{1,2}}$, then we obtain
\begin{eqnarray}\label{eq:1.1}
\!\!\!\!\!\!\!\!&&\|X^\varepsilon_\lambda(t)-X^\varepsilon_{\lambda'}(t)\|^2_{F^*_{1,2}}\nonumber\\
\!\!\!\!\!\!\!\!&&+2\int_0^t\big\langle\Psi(X^\varepsilon_\lambda(s))-\Psi(X^\varepsilon_{\lambda'}(s))+\lambda X^\varepsilon_\lambda(s)-\lambda'X^\varepsilon_{\lambda'}(s),X^\varepsilon_\lambda(s)-X^\varepsilon_{\lambda'}(s)\big\rangle_2ds\nonumber\\
=\!\!\!\!\!\!\!\!&&2\int_0^t\big\langle (1-\varepsilon)\big(\Psi(X^\varepsilon_\lambda(s))-\Psi(X^\varepsilon_{\lambda'}(s))+\lambda X^\varepsilon_\lambda(s)-\lambda'X^\varepsilon_{\lambda'}(s)\big),X^\varepsilon_\lambda(s)-X^\varepsilon_{\lambda'}(s)\big\rangle_{F^*_{1,2}}ds\nonumber\\
\!\!\!\!\!\!\!\!&&+\int_0^t\int_{Z}\big\|f(s,X^\varepsilon_\lambda(s-),z)-f(s,X^\varepsilon_{\lambda'}(s-),z)\big\|^2_{F^*_{1,2}}N(ds,dz)\nonumber\\
\!\!\!\!\!\!\!\!&&+2\int_0^t\!\!\int_{Z}\!\!\big\langle X^\varepsilon_\lambda(s-)-X^\varepsilon_{\lambda'}(s-),f(s,X^\varepsilon_\lambda(s-),z)-f(s,X^\varepsilon_{\lambda'}(s-),z)\big\rangle_{F^*_{1,2}}\widetilde{N}(ds,dz),
\end{eqnarray}
compared with \eqref{eqn18}, we need to estimate one more term
$$2\int_0^t\big\langle (1-\varepsilon)\big(\Psi(X^\varepsilon_\lambda(s))-\Psi(X^\varepsilon_{\lambda'}(s))+\lambda X^\varepsilon_\lambda(s)-\lambda'X^\varepsilon_{\lambda'}(s)\big),X^\varepsilon_\lambda(s)-X^\varepsilon_{\lambda'}(s)\big\rangle_{F^*_{1,2}}ds.$$
However, we cannot find an appropriate method to estimate this term to obtain the desired inequality as \eqref{eqn25}. Similar arguments are used in Claim \ref{claim3}, Proposition \ref{Th2}, and also \cite{BRR, RWX, WZ}.
\end{remark}

Let $H$ be a separable Hilbert space with inner product $\langle\cdot,\cdot\rangle_H$ and $H^*$ its dual. Let $V$ be a reflexive Banach space such that $V\subset H$ continuously and densely. Then for its dual space $V^*$ it follows that $H^*\subset V^*$ continuously and densely. Identifying $H$ and $H^*$ via the Riesz isomorphism we have that
$$V\subset H\subset V^*$$
continuously and densely. If $_{V^*}\langle\cdot,\cdot\rangle_V$ denotes the dualization between $V^*$ and $V$ (i.e. $_{V^*}\langle z,v\rangle_V:=z(v)$ for $z\in V^*, v\in V$), it follows that
\begin{eqnarray}\label{eqn6}
_{V^*}\langle z,v\rangle_V=\langle z,v\rangle_H,~~\text{for~all}~z\in H,~~v\in V.
\end{eqnarray}
$(V,H,V^*)$ is called a Gelfand triple.

In \cite{RWX}, the authors constructed a Gelfand triple with $V=L^2(\mu)$ and $H=F^*_{1,2}$, the Riesz map which identifies $F_{1,2}$ and $F^*_{1,2}$ is $(1-L)^{-1}: F^*_{1,2}\rightarrow F_{1,2}$.

We need the following lemma which was proved in \cite{RWX}.
\begin{lemma}
The map
$$1-L:F_{1,2}\rightarrow F_{1,2}^*$$
extends to a linear isometry
$$1-L:L^2(\mu)\rightarrow(L^2(\mu))^*,$$
and for all $u,v\in L^2(\mu)$,
\begin{eqnarray}\label{eqn7}
_{(L^2(\mu))^*}\langle(1-L)u, v\rangle_{L^2(\mu)}=\int_Eu\cdot
v~d\mu.
\end{eqnarray}
\end{lemma}
We also need the following result (see \cite[Theorem 23.12]{K}).
\begin{proposition}\label{proposition3}\textbf{(Burkholder-Davis-Gundy)}
There exists some constant $c_p\in(0,\infty)$, $p\geq1$, such that for any real-valued local martingale $M=\{M_t, t\in[0,T]\}$ with $M_0=0$,
\begin{eqnarray*}
c_p^{-1}\Bbb{E}\big([M]_T^{\frac{p}{2}}\big)\leq\Bbb{E}\big[\sup_{t\in[0,T]}|M_t|^p\big]\leq c_p\Bbb{E}\big([M]_T^{\frac{p}{2}}\big).
\end{eqnarray*}
Here $[M]=\{[M]_t,t\in[0,T]\}$ is the quadratic variation process of $M$.
\end{proposition}

Denote $\mathcal{H}$ be a Banach space. Throughout the paper, let $L^2([0,T]\times\Omega;\mathcal{H})$ denote
the space of all $\mathcal{H}$-valued square-integrable functions on
$[0,T]\times\Omega$, $L^\infty([0,T],\mathcal{H})$ the space of all $\mathcal{H}$-valued measurable functions on $[0,T]$, $C([0,T];\mathcal{H})$ the space of all continuous $\mathcal{H}$-valued functions on $[0,T]$, $D([0,T];\mathcal{H})$ the space of all c\`{a}dl\`{a}g $\mathcal{H}$-valued functions on $[0,T]$. For two Hilbert spaces $H_1$ and $H_2$, the space of Hilbert-Schmidt operators from $H_1$ to $H_2$ is denoted by $L_2(H_1,H_2)$. For simplicity, the positive constants $c$, $C$, $C_i$, $i=1,2,3$, used in this paper may change from line to line.

\section{Hypothesis and main result}
\setcounter{equation}{0}
 \setcounter{definition}{0}

In this paper, we study \eqref{eq:1} with the following hypotheses:

\medskip
\noindent \textbf{(H1)} $\Psi(\cdot):\Bbb{R}\rightarrow \Bbb{R}$ is a monotonically nondecreasing Lipschitz function with $\Psi(0)=0$.

\vspace{1mm}
\medskip
\noindent \textbf{(H2)} Suppose there exists a positive constant $C_1$ such that
$$\int_{Z}\|f(t,u,z)\|_{F^*_{1,2}}^2\nu(dz)\leq C_1(1+\|u\|^2_{F^*_{1,2}}), \ \forall u\in F^*_{1,2}.$$

\vspace{1mm}
\medskip
\noindent \textbf{(H3)} Suppose there exists a positive constant $C_2$ such that
$$\int_{Z}\|f(t,u_1,z)-f(t,u_2,z)\|_{F^*_{1,2}}^2\nu(dz)\leq C_2\|u_1-u_2\|^2_{F^*_{1,2}}, \ \forall u_1, u_2\in F^*_{1,2}.$$

\vspace{1mm}
\medskip
\noindent \textbf{(H4)} Suppose there exists a positive constant $C_3$ such that
$$\int_{Z}|f(t,u,z)|_{2}^2\nu(dz)\leq C_3(1+|u|^2_{2}), \ \forall u\in L^2(\mu).$$

\begin{remark}
An important physical example of $\Psi$ is the Stefan problem. Consider
\begin{eqnarray}
\Psi(r)=
\left\{ \begin{aligned}
&ar,\ \text{for}\ r<0,\\
&0,\ \text{~for~} 0\leq r\leq \rho,\\
&b(r-\rho),\ \text{~for~} r>\rho,\\
\end{aligned} \right.
\end{eqnarray}
where $a,b,\rho>0$. Notice that $\Psi$ fulfills \textbf{(H1)} with Lipschitz constant $Lip\Psi=\max\{a,b\}$. We omit the details of the Stefan problem, but refer to (\cite[Section 1.1.1]{BDR}) and references therein for related backgrounds.

\end{remark}

\begin{definition}
An $F^*_{1,2}$-valued c\`{a}dl\`{a}g $\mathcal{F}_t$-adapted process $\{X(t)\}_{t\in[0,T]}$ is called a strong solution to \eqref{eq:1}, if
we have
\begin{eqnarray}\label{eqn1}
{X}\in L^2([0,T]\times \Omega; L^2(\mu))\cap L^2(\Omega;L^\infty([0,T];F^*_{1,2}));
\end{eqnarray}

\begin{eqnarray}\label{eqn2}
\int_0^\cdot \Psi({X}(s))ds\in C([0,T];F_{1,2}),\ \Bbb{P}\text{-a.s}.;
\end{eqnarray}

\begin{eqnarray}\label{eqn3}
X(t)=x+L\int_0^t\Psi({X}(s))ds+\int_0^t\int_{Z}f(s,{X}(s-),z)\widetilde{N}(ds,dz),\ \forall t\in[0,T],\ \Bbb{P}\text{-a.s}..
\end{eqnarray}
\end{definition}

Now we can present the main result of this paper.
\begin{theorem}\label{Th1}
Suppose that \textbf{(H1)}-\textbf{(H4)} hold.
Then, for each $x\in L^2(\mu)$, there is a unique strong solution $X$ to
\eqref{eq:1} and exist $C_1, C_2\in(0,\infty)$ satisfying
\begin{eqnarray}\label{eqn4}
\Bbb{E}\Big[\sup_{t\in[0,T]}|X(t)|_2^2\Big]\leq e^{C_1T}(2|x|_2^2+C_2).
\end{eqnarray}
Suppose that \textbf{(H1)}-\textbf{(H3)} and the following hold,
\begin{equation}\label{eq:2}
\Psi(r)r\geq c r^2,\ \forall r\in \mathbb{R},
\end{equation}
where $c\in(0, \infty)$. Then, for all $x\in F_{1,2}^*$, there is a unique strong solution $X$ to
\eqref{eq:1}.
\end{theorem}

\begin{remark}
Suppose $W$ is a cylindrical Wiener process on $L^2(\mu)$, $B: [0, T]\times F^*_{1,2}\times \Omega\rightarrow L_2(L^2(\mu), F^*_{1,2})$ is progressively measurable. If we add a stochastic term of the type $B(t,X(t))dW(t)$ to the right hand side of \eqref{eq:1}, and assume $B(\cdot,u,\cdot)$ satisfies Lipschitz and linear growth conditions w.r.t. $u\in F^*_{1,2}$. Then, Theorem \ref{Th1} continues to hold. For simplicity, in this paper we concentrate on the jump part of the noise.
\end{remark}
\vspace{2mm}
The proof of Theorem \ref{Th1} is given in Section 5.

\section{Approximations}

To prove Theorem \ref{Th1}, we first consider the following approximating equations for \eqref{eq:1}:
\begin{equation} \label{eq:3}
\left\{ \begin{aligned}
&dX^\varepsilon(t)+(\varepsilon-L)\Psi(X^\varepsilon(t))dt=\int_{Z}f(t,X^\varepsilon(t-),z)\widetilde{N}(dt,dz),\ \text{in}\ [0,T]\times E,\\
&X^\varepsilon(0)=x \text{~on~} E,
\end{aligned} \right.
\end{equation}
where $\varepsilon\in(0,1)$. We have the following proposition for \eqref{eq:3}.

\begin{proposition}\label{Th2}
Suppose that \textbf{(H1)}-\textbf{(H4)} hold. Then, for each $x\in L^2(\mu)$, there is a unique $\mathcal{F}_t$-adapted strong solution to \eqref{eq:3}, denoted by $X^\varepsilon$, i.e., it has the following properties,
\begin{eqnarray}\label{eqn5}
X^\varepsilon\in L^2([0,T]\times \Omega; L^2(\mu))\cap L^2(\Omega;L^\infty([0,T];F^*_{1,2}));
\end{eqnarray}
\begin{eqnarray}\label{eqn5.1}
X^\varepsilon\in D([0,T];F^*_{1,2}),\ \Bbb{P}\text{-}a.s.;
\end{eqnarray}
\begin{eqnarray}\label{eqn6}
X^\varepsilon(t)+\!(\varepsilon-L)\!\int_0^t\!\Psi({X^\varepsilon}(s))ds\!=\!x\!+\!\int_0^t\!\int_{Z}\!f(s,{X^\varepsilon}(s-),z)\widetilde{N}(ds,dz),\ \forall t\in[0,T],\ \Bbb{P}\text{-}a.s..
\end{eqnarray}
Furthermore, there exist $C_1, C_2\in(0,\infty)$ such that for all $\varepsilon\in(0,1)$,
\begin{eqnarray}\label{eqn7}
\Bbb{E}\Big[\sup_{t\in[0,T]}|X^\varepsilon(t)|_2^2\Big]\leq e^{C_1T}(2|x|_2^2+C_2).
\end{eqnarray}
Suppose that \textbf{(H1)}-\textbf{(H3)} and \eqref{eq:2} are satisfied. Then, for all $x\in F^*_{1,2}$, there is a unique strong solution to \eqref{eq:3} satisfying \eqref{eqn5}, \eqref{eqn5.1} and \eqref{eqn6}.
\end{proposition}
\begin{proof}
We proceed in two steps. First, we consider the case with initial date $x\in F^*_{1,2}$ and that \eqref{eq:2} is satisfied. Then, we will prove the existence and uniqueness result with $x\in L^2(\mu)$ and without assumption \eqref{eq:2}, by replacing $\Psi$ with $\Psi+\lambda I$, $\lambda\in(0,1)$ and letting $\lambda\rightarrow0$.

\textbf{Step 1}: Assume $x\in F^*_{1,2}$ and that \eqref{eq:2} is satisfied. Set $V:= L^2(\mu)$, $H:=F^*_{1,2}$, $Au:=(L-\varepsilon)\Psi(u)$ for $u\in V$. Under the Gelfand triple $L^2(\mu)\subset F^*_{1,2}\equiv F_{1,2}\subset (L^2(\mu))^*$, we can check the four conditions in \cite[Theorem 1.2]{BLZ} to get the existence and uniqueness of solutions to \eqref{eq:3}.

To make it more explicitly, the hemicontinuity follows directly from \cite[P213, Step 1, (i)]{RWX}, i.e., for all $u, v, w\in V=L^2(\mu)$, for $\iota\in\Bbb{R}$, $|\iota|\leq1$,
\begin{eqnarray*}
\lim_{\iota\rightarrow0}\ _{V^*}\!\langle A(u+\iota v),
w\rangle_V-_{V^*}\!\!\langle Au, w\rangle_V=0.
\end{eqnarray*}

From \cite[Step1, (ii)]{RWX} and \textbf{(H3)}, we can see that the local monotonicity holds, i.e., for all $u_1, u_2\in V(:=L^2(\mu))$,
\begin{eqnarray}\label{eqn8}
&&2_{V^*}\langle Au_1-Au_2, u_1-u_2\rangle_V+\int_{Z}\|f(t,u_1,z)-f(t,u_2,z)\|^2_{F^*_{1,2}}\nu(dz)\nonumber\\
\leq\!\!\!\!\!\!\!\!&&\Big(\frac{2(1-\varepsilon)^2}{\widetilde{\alpha}}+C_2\Big)\|u_1-u_2\|^2_{F^*_{1,2}}.
\end{eqnarray}
Here $\widetilde{\alpha}:=(k+1)^{-1}$, $k:=Lip\Psi$, which is the Lipschitz constant of $\Psi$.
\vspace{2mm}
As a short remark, the difference of the estimates between \cite[(3.9)]{RWX} and \eqref{eqn8} only lies in the second term in both left-hand sides. Since both terms satisfy Lipschitz condition, the local monotonicity is obvious.

From \cite[Step1, (iii)]{RWX} and \textbf{(H2)}, we can see easily that the coercivity holds, i.e., for all $u\in V(:=L^2(\mu))$,
$$2_{V^*}\langle Au, u\rangle_V\leq \big[-2c+2\theta^2k^2(1-\varepsilon)\big]\cdot|u|_2^2+\Big[\frac{2(1-\varepsilon)}{\theta^2}+C_1\Big]\cdot\|u\|^2_{F^*_{1,2}}+C_1.$$
Here $\theta$ is a positive constant and small enough such that $-2c+2\theta^2k^2(1-\varepsilon)$ is negative.

\cite[Step1, (iv)]{RWX} implies the growth condition, i.e., for all $u\in V(:=L^2(\mu))$,
$$|Au|_{V^*}\leq 2k|u|_2.$$

By applying \cite[Theorem 1.2]{BLZ}, we know that there exists a unique strong solution to \eqref{eq:3}, denoted by $X^\varepsilon$.

\textbf{Step 2}: If $\Psi$ doesn't satisfy \eqref{eq:2}, the hemicontinuity, local monotonicity and growth conditions still hold, but the coercivity condition not in general. In this case, we will approximate $\Psi$ by $\Psi+\lambda I$, $\lambda\in(0,1)$.

Consider the following approximating equations for \eqref{eq:3}:
\begin{equation} \label{eq:4}
\left\{ \begin{aligned}
&dX^{\varepsilon}_{\lambda}(t)+(\varepsilon-L)\big(\Psi(X^\varepsilon_\lambda(t)+\lambda X^\varepsilon_\lambda(t)\big)dt=\int_{Z}f(t,X^\varepsilon_\lambda(t-),z)\widetilde{N}(dt,dz),\ \text{in}\ [0,T]\times E,\\
&X^\varepsilon_\lambda(0)=x\in F^*_{1,2} \text{~on~} E.
\end{aligned} \right.
\end{equation}
By using the similar argument as in Step 1, it is easy to prove that there exists a unique strong solution to \eqref{eq:4} which satisfies, $X^\varepsilon_\lambda\in D([0,T];F^*_{1,2})$, $\Bbb{P}$-a.s., $X^\varepsilon_\lambda\in L^2([0,T]\times \Omega; L^2(\mu))\cap L^2(\Omega;L^\infty([0,T];F^*_{1,2}))$, and
\begin{eqnarray*}
X^\varepsilon_\lambda(t)+\!\!(\varepsilon-L)\int_0^t\Psi(X^\varepsilon_\lambda(s))ds=x+\!\!\int_0^t\int_{Z}f(s,X^\varepsilon_\lambda(s-),z)\widetilde{N}(ds,dz),\ \forall t\in[0,T],\ \Bbb{P}\text{-a.s}.,
\end{eqnarray*}
and
\begin{eqnarray}\label{eqn9}
\Bbb{E}\Big[\sup_{t\in[0,T]}\|X^\varepsilon_\lambda(t)\|^2_{F^*_{1,2}}\Big]<\infty.
\end{eqnarray}

In the following, we want to prove that the sequence $\{X^\varepsilon_\lambda\}_{\lambda\in(0,1)}$ converges to the solution of \eqref{eq:3} as $\lambda\rightarrow0$. From now on, we assume that the initial date $x\in L^2(\mu)$. This proof is divided into three parts, which are given as three claims.

\begin{claim}\label{claim1}
\begin{eqnarray*}
\Bbb{E}\Big[\sup_{s\in[0,T]}|X^\varepsilon_\lambda(s)|_2^2\Big]+4\lambda\Bbb{E}\int_0^T\|X^\varepsilon_\lambda(s)\|^2_{F_{1,2}}ds\leq e^{C_1T}(2|x|_2^2+C_2),\ \forall \varepsilon, \lambda\in(0,1).
\end{eqnarray*}
\end{claim}
\begin{proof}
Rewrite \eqref{eq:4} as follows, for all $t\in[0,T]$,
\begin{eqnarray}\label{eqn10}
X^\varepsilon_\lambda(t)=x+\!\!\int_0^t(L-\varepsilon)\big(\Psi(X^\varepsilon_\lambda(s))+\lambda X^\varepsilon_\lambda(s)\big)ds+\!\!\int_0^t\int_{Z}f(s,X^\varepsilon_\lambda(s-),z)\widetilde{N}(ds,dz).
\end{eqnarray}
Let $\varepsilon<\alpha<\infty$, applying the operator $(\alpha-L)^{-\frac{1}{2}}:F^*_{1,2}\rightarrow L^2(\mu)$ to both sides of \eqref{eqn10}, we obtain
\begin{eqnarray}\label{eqn11}
&&\!\!\!\!\!\!\!\!(\alpha-L)^{-\frac{1}{2}}X^\varepsilon_\lambda(t)\nonumber\\
=&&\!\!\!\!\!\!\!\!(\alpha-L)^{-\frac{1}{2}}x+\int_0^t(L-\varepsilon)(\alpha-L)^{-\frac{1}{2}}\big(\Psi(X^\varepsilon_\lambda(s))+\lambda X^\varepsilon_\lambda(s)\big)ds\nonumber\\
&&\!\!\!\!\!\!\!\!+\int_0^t\int_{Z}(\alpha-L)^{-\frac{1}{2}}f(s,X^\varepsilon_\lambda(s-),z)\widetilde{N}(ds,dz).
\end{eqnarray}
Applying It\^{o}'s formula (cf, \cite[Remark A.2]{BHZ}) to $\big|(\alpha-L)^{-\frac{1}{2}}X^\varepsilon_\lambda(t)\big|_2^2$ with $H=L^2(\mu)$, $V=F_{1,2}$, we obtain that for $t\in[0,T]$,
\begin{eqnarray}\label{eqn12}
&&\!\!\!\!\!\!\!\!\big|(\alpha-L)^{-\frac{1}{2}}X^\varepsilon_\lambda(t)\big|_2^2\nonumber\\
=&&\!\!\!\!\!\!\!\!\big|(\alpha-L)^{-\frac{1}{2}}x\big|_2^2+2\int_0^t\ _{F^*_{1,2}}\big\langle(L-\varepsilon)(\alpha-L)^{-\frac{1}{2}}\Psi(X^\varepsilon_\lambda(s)),(\alpha-L)^{-\frac{1}{2}}X^\varepsilon_\lambda(s)\big\rangle_{F_{1,2}}ds\nonumber\\
&&\!\!\!\!\!\!\!\!+2\lambda\int_0^t\ _{F^*_{1,2}}\big\langle(L-\varepsilon)(\alpha-L)^{-\frac{1}{2}} X^\varepsilon_\lambda(s),(\alpha-L)^{-\frac{1}{2}}X^\varepsilon_\lambda(s)\big\rangle_{F_{1,2}}ds\nonumber\\
&&\!\!\!\!\!\!\!\!+2\int_0^t\int_{Z}\big\langle(\alpha-L)^{-\frac{1}{2}}X^\varepsilon_\lambda(s-),(\alpha-L)^{-\frac{1}{2}}f(s,X^\varepsilon_\lambda(s-),z)\big\rangle_2\widetilde{N}(ds,dz)\nonumber\\
&&\!\!\!\!\!\!\!\!+\int_0^t\int_{Z}\big|f(s,X^\varepsilon_\lambda(s-),z)\big|_2^2N(ds,dz).
\end{eqnarray}
From \cite[(3.19)]{RWX2022}, we know that the second term in the right-hand side of \eqref{eqn12} is non-positive. From \cite[(3.20)]{RWX2022}, we know that the third term in the right-hand side of \eqref{eqn12} can be dominated by
\begin{eqnarray*}
&&\!\!\!\!\!\!\!\!2\lambda\int_0^t\ _{F^*_{1,2}}\big\langle(L-\varepsilon)(\alpha-L)^{-\frac{1}{2}} X^\varepsilon_\lambda(s),(\alpha-L)^{-\frac{1}{2}}X^\varepsilon_\lambda(s)\big\rangle_{F_{1,2}}ds\nonumber\\
\leq&&\!\!\!\!\!\!\!\!-2\lambda\int_0^t\big\|(\alpha-L)^{-\frac{1}{2}}X^\varepsilon_\lambda(s)\big\|_{F_{1,2}}^2ds+2\int_0^t\big|(\alpha-L)^{-\frac{1}{2}}X^\varepsilon_\lambda(s)\big|^2_2ds.
\end{eqnarray*}
Multiplying both sides of \eqref{eqn12} by $\alpha$, using the above estimates and taking into account that $\sqrt{\alpha}(\alpha-L)^{-\frac{1}{2}}$ is a contraction on $L^2(\mu)$ (\cite[(3.22)]{RWX2022}), \eqref{eqn12} yields that for all $t\in[0,T]$,
\begin{eqnarray}\label{eqn13}
&&\big|\sqrt{\alpha}(\alpha-L)^{-\frac{1}{2}}X^\varepsilon_\lambda(t)\big|_2^2+2\lambda\int_0^t\big\|\sqrt{\alpha}(\alpha-L)^{-\frac{1}{2}}X^\varepsilon_\lambda(s)\big\|^2_{F_{1,2}}ds\nonumber\\
\le\!\!\!\!\!\!\!\!&&\big|\sqrt{\alpha}(\alpha-L)^{-\frac{1}{2}}x\big|_2^2+2\int_0^t\big|(\alpha-L)^{-\frac{1}{2}}X^\varepsilon_\lambda(s)\big|^2_2ds\nonumber\\
\!\!\!\!\!\!\!\!&&+2\int_0^t\int_{Z}\big\langle\sqrt{\alpha}(\alpha-L)^{-\frac{1}{2}}X^\varepsilon_\lambda(s-),\sqrt{\alpha}(\alpha-L)^{-\frac{1}{2}}f(s,X^\varepsilon_\lambda(s-),z)\big\rangle_2\widetilde{N}(ds,dz)\nonumber\\
\!\!\!\!\!\!\!\!&&+\int_0^t\int_{Z}\big|\sqrt{\alpha}(\alpha-L)^{-\frac{1}{2}}f(s,X^\varepsilon_\lambda(s-),z)\big|_2^2N(ds,dz)\nonumber\\
:=\!\!\!\!\!\!\!\!&&\big|\sqrt{\alpha}(\alpha-L)^{-\frac{1}{2}}x\big|_2^2+2\int_0^t|X^\varepsilon_\lambda(s)|^2_2ds+I_1(t)+I_2(t).
\end{eqnarray}
If $\varepsilon<\alpha<1$, by \eqref{eqn22} and \textbf{(H2)}, we have that for all $t\in[0,T]$,
\begin{eqnarray}\label{eqn14}
\!\!\!\!\!\!\!\!&&\Bbb{E}\int_0^t\int_{Z}\big|\sqrt{\alpha}(\alpha-L)^{-\frac{1}{2}}f(s,X^\varepsilon_\lambda(s),z)\big|_2^2\nu(dz)ds\nonumber\\
=\!\!\!\!\!\!\!\!&&\Bbb{E}\int_0^t\int_{Z}~ _{F_{1,2}}\langle \alpha(\alpha-L)^{-1}f(s,X^\varepsilon_\lambda(s),z),f(s,X^\varepsilon_\lambda(s),z)\rangle_{F_{1,2}^*}\nu(dz)ds\nonumber\\
=\!\!\!\!\!\!\!\!&&\Bbb{E}\int_0^t\int_{Z}\alpha\|f(s,X^\varepsilon_\lambda(s),z)\|^2_{F^*_{1,2,\alpha}}\nu(dz)ds\nonumber\\
\leq\!\!\!\!\!\!\!\!&&\Bbb{E}\int_0^t\int_{Z}\|f(s,X^\varepsilon_\lambda(s),z)\|^2_{F^*_{1,2}}\nu(dz)ds\nonumber\\
\leq\!\!\!\!\!\!\!\!&&\Bbb{E}\int_0^tC_1(1+\|X^\varepsilon_\lambda(s)\|^2_{F_{1,2}^*})ds.
\end{eqnarray}
Similarly, if $1\leq\alpha<\infty$, by \eqref{equivalent norm 2} and \textbf{(H2)}, we have that for all $t\in[0,T]$,
\begin{eqnarray}\label{eqn14.1}
\!\!\!\!\!\!\!\!&&\Bbb{E}\int_0^t\int_{Z}\big|\sqrt{\alpha}(\alpha-L)^{-\frac{1}{2}}f(s,X^\varepsilon_\lambda(s),z)\big|_2^2\nu(dz)ds\nonumber\\
\leq\!\!\!\!\!\!\!\!&&\Bbb{E}\int_0^tC_1\alpha(1+\|X^\varepsilon_\lambda(s)\|^2_{F_{1,2}^*})ds.
\end{eqnarray}
Therefore, from \eqref{eqn14}, \eqref{eqn14.1} and \eqref{eqn9}, we know that for $\varepsilon<\alpha<\infty$, $t\in[0,T]$,
$$\Bbb{E}\int_0^t\int_{Z}\big|\sqrt{\alpha}(\alpha-L)^{-\frac{1}{2}}f(s,X^\varepsilon_\lambda(s),z)\big|_2^2\nu(dz)ds<\infty,$$
then by the equality under \cite[page:62, Chapter II, (3.7)]{IW}, we have
\begin{eqnarray}\label{eqn14.2}
\!\!\!\!\!\!\!\!&&\Bbb{E}\big[I_2(t)\big]\nonumber\\
=\!\!\!\!\!\!\!\!&&\Bbb{E}\int_0^t\int_{Z}\big|\sqrt{\alpha}(\alpha-L)^{-\frac{1}{2}}f(s,X^\varepsilon_\lambda(s-),z)\big|_2^2N(ds,dz)\nonumber\\
=\!\!\!\!\!\!\!\!&&\Bbb{E}\int_0^t\int_{Z}\big|\sqrt{\alpha}(\alpha-L)^{-\frac{1}{2}}f(s,X^\varepsilon_\lambda(s),z)\big|_2^2\nu(dz)ds.
\end{eqnarray}
Since $\sqrt{\alpha}(\alpha-L)^{-\frac{1}{2}}$ is a contraction on $L^2(\mu)$, \eqref{eqn14.2} and \textbf{(H4)} implies,
\begin{eqnarray}\label{eqn14.3}
\!\!\!\!\!\!\!\!&&\Bbb{E}\big[I_2(t)\big]\nonumber\\
=\!\!\!\!\!\!\!\!&&\Bbb{E}\int_0^t\int_{Z}\big|\sqrt{\alpha}(\alpha-L)^{-\frac{1}{2}}f(s,X^\varepsilon_\lambda(s),z)\big|_2^2\nu(dz)ds\nonumber\\
\leq\!\!\!\!\!\!\!\!&&\Bbb{E}\int_0^t\int_{Z}\big|f(s,X^\varepsilon_\lambda(s),z)\big|^2_{2}\nu(dz)ds\nonumber\\
\leq\!\!\!\!\!\!\!\!&&C_3+C_3\Bbb{E}\int_0^t\big|X^\varepsilon_\lambda(s)\big|^2_{2}ds.
\end{eqnarray}
Note that $I_1:=\{I_1(t), t\in[0,T]\}$ is a real-valued c\`{a}dl\`{a}g local martingale with respect to $\Bbb{F}$, and \cite[Page:129, Theorem 8.23(iv)]{PZ} implies that the quadratic variation process of $I_1$ is
\begin{eqnarray*}
[I_1]_t=\int_0^t\int_{Z}\big|\sqrt{\alpha}(\alpha-L)^{-\frac{1}{2}}f(s,X^\varepsilon_\lambda(s-),z)\big|_2^2N(ds,dz),~t\in[0,T],
\end{eqnarray*}
then by using the Burkhold-Davis-Gundy (BDG) inequality (with $p$=1, see Proposition \ref{proposition3}), H\"{o}lder inequality, Young inequality and \eqref{eqn14.3}, we obtain that for all $t\in[0,T]$,
\begin{eqnarray}\label{eqn15}
\!\!\!\!\!\!\!\!&&\Bbb{E}\big[\sup_{s\in[0,t]}|I_1(s)|\big]\nonumber\\
\leq\!\!\!\!\!\!\!\!&&C\Bbb{E}\Bigg[\int_0^t\int_{Z}\Big|\big\langle\sqrt{\alpha}(\alpha-L)^{-\frac{1}{2}}X^\varepsilon_\lambda(s-),\sqrt{\alpha}(\alpha-L)^{-\frac{1}{2}}f(s,X^\varepsilon_\lambda(s-),z)\big\rangle_2\Big|^2N(ds,dz)\Bigg]^{\frac{1}{2}}\nonumber\\
\leq\!\!\!\!\!\!\!\!&&C\Bbb{E}\Bigg[\sup_{s\in[0,t]}\big|\sqrt{\alpha}(\alpha-L)^{-\frac{1}{2}}X^\varepsilon_\lambda(s)\big|_2^2\cdot\Big(\int_0^t\int_{Z}\big|\sqrt{\alpha}(\alpha-L)^{-\frac{1}{2}}f(s,X^\varepsilon_\lambda(s-),z)\big|_2^2N(ds,dz)\Big)\Bigg]^{\frac{1}{2}}\nonumber\\
\leq\!\!\!\!\!\!\!\!&&C\Bigg[\Bbb{E}\sup_{s\in[0,t]}\big|\sqrt{\alpha}(\alpha-L)^{-\frac{1}{2}}X^\varepsilon_\lambda(s)\big|_2^2\Bigg]^{\frac{1}{2}}\cdot\Bigg[\Bbb{E}\int_0^t\int_{Z}\big|\sqrt{\alpha}(\alpha-L)^{-\frac{1}{2}}f(s,X^\varepsilon_\lambda(s-),z)\big|_2^2N(ds,dz)\Bigg]^{\frac{1}{2}}\nonumber\\
=\!\!\!\!\!\!\!\!&&C\Bigg[\Bbb{E}\sup_{s\in[0,t]}\big|\sqrt{\alpha}(\alpha-L)^{-\frac{1}{2}}X^\varepsilon_\lambda(s)\big|_2^2\Bigg]^{\frac{1}{2}}\cdot\Bigg[\Bbb{E}\int_0^t\int_{Z}\big|\sqrt{\alpha}(\alpha-L)^{-\frac{1}{2}}f(s,X^\varepsilon_\lambda(s),z)\big|_2^2\nu(dz)ds\Bigg]^{\frac{1}{2}}\nonumber\\
\leq\!\!\!\!\!\!\!\!&&\frac{1}{2}\Bbb{E}\Big[\sup_{s\in[0,t]}\big|\sqrt{\alpha}(\alpha-L)^{-\frac{1}{2}}X^\varepsilon_\lambda(s)\big|_2^2\Big]+C\Bbb{E}\int_0^t\int_{Z}\big|\sqrt{\alpha}(\alpha-L)^{-\frac{1}{2}}f(s,X^\varepsilon_\lambda(s),z)\big|_2^2\nu(dz)ds\nonumber\\
\leq\!\!\!\!\!\!\!\!&&\frac{1}{2}\Bbb{E}\Big[\sup_{s\in[0,t]}\big|\sqrt{\alpha}(\alpha-L)^{-\frac{1}{2}}X^\varepsilon_\lambda(s)\big|_2^2\Big]+C_2\Bbb{E}\int_0^t\big|X^\varepsilon_\lambda(s)\big|^2_{2}ds+C_2.
\end{eqnarray}
By \eqref{eqn13}, \eqref{eqn14.3} and \eqref{eqn15}, we obtain that, for all $t\in[0,T]$,
\begin{eqnarray}\label{eqn16}
\!\!\!\!\!\!\!\!&&\Bbb{E}\Big[\sup_{s\in[0,t]}\big|\sqrt{\alpha}(\alpha-L)^{-\frac{1}{2}}X^\varepsilon_\lambda(s)\big|_2^2\Big]+2\lambda\int_0^t\big\|\sqrt{\alpha}(\alpha-L)^{-\frac{1}{2}}X^\varepsilon_\lambda(s)\big\|_{F_{1,2}}^2ds\nonumber\\
\leq\!\!\!\!\!\!\!\!&&\big|\sqrt{\alpha}(\alpha-L)^{-\frac{1}{2}}x\big|_2^2+C_1\Bbb{E}\int_0^t|X^\varepsilon_\lambda(s)|^2_2ds+C_2\nonumber\\
\!\!\!\!\!\!\!\!&&+\frac{1}{2}\Bbb{E}\Big[\sup_{s\in[0,t]}\big|\sqrt{\alpha}(\alpha-L)^{-\frac{1}{2}}X^\varepsilon_\lambda(s)\big|_2^2\Big].
\end{eqnarray}
Since $|\sqrt{\alpha}(\alpha-L)^{-\frac{1}{2}}\cdot|_2$ is equivalent to $\|\cdot\|_{F_{1,2}}^2$, we know the first summand of the left-hand side of \eqref{eqn16} is finite by \eqref{eqn9}, \eqref{eqn16} shows that
\begin{eqnarray}\label{eqn17}
\!\!\!\!\!\!\!\!&&\Bbb{E}\Big[\sup_{s\in[0,t]}\big|\sqrt{\alpha}(\alpha-L)^{-\frac{1}{2}}X^\varepsilon_\lambda(s)\big|_2^2\Big]+4\lambda\int_0^t\big\|\sqrt{\alpha}(\alpha-L)^{-\frac{1}{2}}X^\varepsilon_\lambda(s)\big\|_{F_{1,2}}^2ds\nonumber\\
\leq\!\!\!\!\!\!\!\!&&2\big|\sqrt{\alpha}(\alpha-L)^{-\frac{1}{2}}x\big|_2^2+C_1\Bbb{E}\int_0^t|X^\varepsilon_\lambda(s)|^2_2ds+C_2.
\end{eqnarray}
Letting $\alpha\rightarrow\infty$, similarly as to get \cite[(3.32)]{RWX2022}, we have that for $t\in[0,T]$,
\begin{eqnarray*}
\Bbb{E}\Big[\sup_{s\in[0,t]}|X^\varepsilon_\lambda(s)|_2^2\Big]+4\lambda\int_0^t\|X^\varepsilon_\lambda(s)\|_{F_{1,2}}^2ds\leq2|x|_2^2+C_1\Bbb{E}\int_0^t|X^\varepsilon_\lambda(s)|^2_2ds+C_2.
\end{eqnarray*}
The Gronwall inequality yields
\begin{eqnarray*}
\Bbb{E}\Big[\sup_{s\in[0,T]}|X^\varepsilon_\lambda(s)|_2^2\Big]+4\lambda\int_0^T\|X^\varepsilon_\lambda(s)\|_{F_{1,2}}^2ds\leq e^{C_1T}(2|x|_2^2+C_2).
\end{eqnarray*}\hspace{\fill}$\Box$
\end{proof}

\begin{claim}\label{claim2}
There exists an $F^*_{1,2}$-valued c\`{a}dl\`{a}g $\mathcal{F}_t$-adapted process $\{X^\varepsilon(t)\}_{t\in[0,T]}$ such that $X^\varepsilon\in L^2(\Omega;L^\infty([0,T];F^*_{1,2}))\cap L^2([0,T]\times\Omega;L^2(\mu))$, and
$$\lim_{\lambda\rightarrow0}\Bbb{E}\Big[\sup_{s\in[0,T]}\|X^\varepsilon_\lambda(s)-X^\varepsilon(s)\|^2_{F^*_{1,2}}\Big]=0.$$
\end{claim}
\begin{proof}
By It\^{o}'s formula we get that, for $\lambda', \lambda\in(0,1)$ and $t\in[0,T]$,
\begin{eqnarray}\label{eqn18}
\!\!\!\!\!\!\!\!&&\|X^\varepsilon_\lambda(t)-X^\varepsilon_{\lambda'}(t)\|^2_{F^*_{1,2,\varepsilon}}\nonumber\\
\!\!\!\!\!\!\!\!&&+2\int_0^t\big\langle\Psi(X^\varepsilon_\lambda(s))-\Psi(X^\varepsilon_{\lambda'}(s))+\lambda X^\varepsilon_\lambda(s)-\lambda'X^\varepsilon_{\lambda'}(s),X^\varepsilon_\lambda(s)-X^\varepsilon_{\lambda'}(s)\big\rangle_2ds\nonumber\\
=\!\!\!\!\!\!\!\!&&\int_0^t\int_{Z}\big\|f(s,X^\varepsilon_\lambda(s-),z)-f(s,X^\varepsilon_{\lambda'}(s-),z)\big\|^2_{F^*_{1,2,\varepsilon}}N(ds,dz)\nonumber\\
\!\!\!\!\!\!\!\!&&+2\int_0^t\!\!\int_{Z}\!\!\big\langle X^\varepsilon_\lambda(s-)-X^\varepsilon_{\lambda'}(s-),f(s,X^\varepsilon_\lambda(s-),z)-f(s,X^\varepsilon_{\lambda'}(s-),z)\big\rangle_{F^*_{1,2,\varepsilon}}\!\!\!\!\widetilde{N}(ds,dz).
\end{eqnarray}
From \cite[(3.27)]{RWX}, we know that
\begin{eqnarray}\label{eqn19}
\!\!\!\!\!\!\!\!&&2\int_0^t\big\langle\Psi(X^\varepsilon_\lambda(s))-\Psi(X^\varepsilon_{\lambda'}(s))+\lambda X^\varepsilon_\lambda(s)-\lambda'X^\varepsilon_{\lambda'}(s),X^\varepsilon_\lambda(s)-X^\varepsilon_{\lambda'}(s)\big\rangle_2ds\nonumber\\
\geq\!\!\!\!\!\!\!\!&&2\tilde{\alpha}\int_0^t\big|\Psi(X^\varepsilon_\lambda(s))-\Psi(X^\varepsilon_{\lambda'}(s))\big|_2^2ds\nonumber\\
\!\!\!\!\!\!\!\!&&+2\int_0^t\big\langle\lambda X^\varepsilon_\lambda(s)-\lambda'X^\varepsilon_{\lambda'}(s),X^\varepsilon_\lambda(s)-X^\varepsilon_{\lambda'}(s)\big\rangle_2ds.
\end{eqnarray}
Taking \eqref{eqn19} into \eqref{eqn18}, we obtain
\begin{eqnarray}\label{eqn20}
\!\!\!\!\!\!\!\!&&\|X^\varepsilon_\lambda(t)-X^\varepsilon_{\lambda'}(t)\|^2_{F^*_{1,2,\varepsilon}}+2\tilde{\alpha}\int_0^t\big|\Psi(X^\varepsilon_\lambda(s))-\Psi(X^\varepsilon_{\lambda'}(s))\big|_2^2ds\nonumber\\
\leq\!\!\!\!\!\!\!\!&&-2\int_0^t\big\langle\lambda X^\varepsilon_\lambda(s)-\lambda'X^\varepsilon_{\lambda'}(s),X^\varepsilon_\lambda(s)-X^\varepsilon_{\lambda'}(s)\big\rangle_2ds\nonumber\\
\!\!\!\!\!\!\!\!&&+\int_0^t\int_{Z}\big\|f(s,X^\varepsilon_\lambda(s-),z)-f(s,X^\varepsilon_{\lambda'}(s-),z)\big\|^2_{F^*_{1,2,\varepsilon}}N(ds,dz)\nonumber\\
\!\!\!\!\!\!\!\!&&+2\int_0^t\!\!\int_{Z}\!\!\big\langle X^\varepsilon_\lambda(s-)-X^\varepsilon_{\lambda'}(s-),f(s,X^\varepsilon_\lambda(s-),z)-f(s,X^\varepsilon_{\lambda'}(s-),z)\big\rangle_{F^*_{1,2,\varepsilon}}\!\!\!\!\widetilde{N}(ds,dz)\nonumber\\
\leq\!\!\!\!\!\!\!\!&&4(\lambda+\lambda')\int_0^t|X^\varepsilon_\lambda(s)|_2^2+|X^\varepsilon_{\lambda'}(s)|_2^2ds\nonumber\\
\!\!\!\!\!\!\!\!&&+\int_0^t\int_{Z}\big\|f(s,X^\varepsilon_\lambda(s-),z)-f(s,X^\varepsilon_{\lambda'}(s-),z)\big\|^2_{F^*_{1,2,\varepsilon}}N(ds,dz)\nonumber\\
\!\!\!\!\!\!\!\!&&+2\int_0^t\!\!\int_{Z}\!\!\big\langle X^\varepsilon_\lambda(s-)-X^\varepsilon_{\lambda'}(s-),f(s,X^\varepsilon_\lambda(s-),z)-f(s,X^\varepsilon_{\lambda'}(s-),z)\big\rangle_{F^*_{1,2,\varepsilon}}\!\!\!\!\widetilde{N}(ds,dz).
\end{eqnarray}
Taking expectation to both sides of \eqref{eqn20}, we obtain,
\begin{eqnarray}\label{eqn21}
\!\!\!\!\!\!\!\!&&\Bbb{E}\Big[\sup_{s\in[0,t]}\|X^\varepsilon_\lambda(s)-X^\varepsilon_{\lambda'}(s)\|^2_{F^*_{1,2,\varepsilon}}\Big]+2\tilde{\alpha}\Bbb{E}\int_0^t\big|\Psi(X^\varepsilon_\lambda(s))-\Psi(X^\varepsilon_{\lambda'}(s))\big|_2^2ds\nonumber\\
\leq\!\!\!\!\!\!\!\!&&4(\lambda+\lambda')\Bbb{E}\int_0^t|X^\varepsilon_\lambda(s)|_2^2+|X^\varepsilon_{\lambda'}(s)|_2^2ds\nonumber\\
\!\!\!\!\!\!\!\!&&+\Bbb{E}\int_0^t\int_{Z}\big\|f(s,X^\varepsilon_\lambda(s),z)-f(s,X^\varepsilon_{\lambda'}(s),z)\big\|^2_{F^*_{1,2,\varepsilon}}\nu(dz)ds\nonumber\\
\!\!\!\!\!\!\!\!&&+2\Bbb{E}\Bigg[\sup_{s\in[0,t]}\Big|\int_0^s\!\!\int_{Z}\!\!\big\langle X^\varepsilon_\lambda(l-)-X^\varepsilon_{\lambda'}(l-),f(l,X^\varepsilon_\lambda(l-),z)-f(l,X^\varepsilon_{\lambda'}(l-),z)\big\rangle_{F^*_{1,2,\varepsilon}}\!\!\!\!\widetilde{N}(dl,dz)\Big|\Bigg]\nonumber\\
:=\!\!\!\!\!\!\!\!&&4(\lambda+\lambda')\Bbb{E}\int_0^t|X^\varepsilon_\lambda(s)|_2^2+|X^\varepsilon_{\lambda'}(s)|_2^2ds+J_1(t)+J_2(t).
\end{eqnarray}
By assumption \textbf{(H3)} and \eqref{eqn22}, we get
\begin{eqnarray}\label{eqn23}
J_1(t)\!\!\!\!\!\!\!\!&&\leq\frac{1}{\sqrt{\varepsilon}}\Bbb{E}\int_0^t\int_{Z}\big\|f(s,X^\varepsilon_\lambda(s),z)-f(s,X^\varepsilon_{\lambda'}(s),z)\big\|^2_{F^*_{1,2}}\nu(dz)ds\nonumber\\
\!\!\!\!\!\!\!\!&&\leq \frac{C}{\sqrt{\varepsilon}}\Bbb{E}\int_0^t\big\|X^\varepsilon_\lambda(s)-X^\varepsilon_{\lambda'}(s)\big\|^2_{F^*_{1,2}}ds\nonumber\\
\!\!\!\!\!\!\!\!&&\leq \frac{C}{\sqrt{\varepsilon}}\Bbb{E}\int_0^t\big\|X^\varepsilon_\lambda(s)-X^\varepsilon_{\lambda'}(s)\big\|^2_{F^*_{1,2,\varepsilon}}ds.
\end{eqnarray}
Using the arguments similar to that used in \eqref{eqn15}, for $t\in[0,T]$, we get
\begin{eqnarray}\label{eqn24}
J_2(t)\leq\!\!\!\!\!\!\!\!&& C\Bbb{E}\Bigg[\Big|\int_0^t\int_{Z}\big\langle X^\varepsilon_\lambda(l-)-X^\varepsilon_{\lambda'}(l-),f(l,X^\varepsilon_\lambda(l-),z)-f(l,X^\varepsilon_{\lambda'}(l-),z)\big\rangle^2_{F^*_{1,2,\varepsilon}}N(dl,dz)\Big|^{\frac{1}{2}}\Bigg]\nonumber\\
\leq \!\!\!\!\!\!\!\!&& \Bbb{E}\Bigg[\Big|\sup_{l\in[0,t]}\|X^\varepsilon_\lambda(l)-X^\varepsilon_{\lambda'}(l)\|^2_{F^*_{1,2,\varepsilon}}\cdot\int_0^t\int_{Z}\|f(l,X^\varepsilon_\lambda(l-),z)-f(l,X^\varepsilon_{\lambda'}(l-),z)\|^2_{F^*_{1,2,\varepsilon}}N(dl,dz)\Big|^{\frac{1}{2}}\Bigg]\nonumber\\
\leq\!\!\!\!\!\!\!\!&&\frac{1}{2}\Bbb{E}\Big[\sup_{l\in[0,t]}\|X^\varepsilon_\lambda(l)-X^\varepsilon_{\lambda'}(l)\|^2_{F^*_{1,2,\varepsilon}}\Big]\nonumber\\
\!\!\!\!\!\!\!\!&&+C\Bbb{E}\int_0^t\int_{Z}\|f(l,X^\varepsilon_\lambda(l),z)-f(l,X^\varepsilon_{\lambda'}(l),z)\|^2_{F^*_{1,2,\varepsilon}}\nu(dz)dl\nonumber\\
\leq\!\!\!\!\!\!\!\!&&\frac{1}{2}\Bbb{E}\Big[\sup_{l\in[0,t]}\|X^\varepsilon_\lambda(l)-X^\varepsilon_{\lambda'}(l)\|^2_{F^*_{1,2,\varepsilon}}\Big]+\frac{C}{\sqrt{\varepsilon}}\Bbb{E}\int_0^t\big\|X^\varepsilon_\lambda(s)-X^\varepsilon_{\lambda'}(s)\big\|^2_{F^*_{1,2,\varepsilon}}ds.
\end{eqnarray}
Taking \eqref{eqn23} and \eqref{eqn24} into \eqref{eqn21}, we get
\begin{eqnarray*}
\!\!\!\!\!\!\!\!&&\Bbb{E}\Big[\sup_{s\in[0,t]}\|X^\varepsilon_\lambda(s)-X^\varepsilon_{\lambda'}(s)\|^2_{F^*_{1,2,\varepsilon}}\Big]+2\tilde{\alpha}\Bbb{E}\int_0^t\big|\Psi(X^\varepsilon_\lambda(s))-\Psi(X^\varepsilon_{\lambda'}(s))\big|_2^2ds\nonumber\\
\leq\!\!\!\!\!\!\!\!&&4(\lambda+\lambda')\Bbb{E}\int_0^t|X^\varepsilon_\lambda(s)|_2^2+|X^\varepsilon_{\lambda'}(s)|_2^2ds+\frac{1}{2}\Bbb{E}\Big[\sup_{l\in[0,t]}\|X^\varepsilon_\lambda(l)-X^\varepsilon_{\lambda'}(l)\|^2_{F^*_{1,2,\varepsilon}}\Big]\nonumber\\
\!\!\!\!\!\!\!\!&&+\frac{C}{\sqrt{\varepsilon}}\Bbb{E}\int_0^t\big\|X^\varepsilon_\lambda(s)-X^\varepsilon_{\lambda'}(s)\big\|^2_{F^*_{1,2,\varepsilon}}ds.
\end{eqnarray*}
Since $x\in L^2(\mu)$, Gronwall's lemma and Claim \ref{claim1} imply that for some constant $C\in(0,\infty)$, which is independent of $\lambda, \lambda', \varepsilon$,
\begin{eqnarray*}
\!\!\!\!\!\!\!\!&&\Bbb{E}\Big[\sup_{s\in[0,t]}\|X^\varepsilon_\lambda(s)-X^\varepsilon_{\lambda'}(s)\|^2_{F^*_{1,2,\varepsilon}}\Big]+4\tilde{\alpha}\Bbb{E}\int_0^t\big|\Psi(X^\varepsilon_\lambda(s))-\Psi(X^\varepsilon_{\lambda'}(s))\big|_2^2ds\nonumber\\
\leq\!\!\!\!\!\!\!\!&&e^{\frac{C}{\sqrt{\varepsilon}}}\cdot 16e^{C_1T}(|x|_2^2+C_2)(\lambda+\lambda').
\end{eqnarray*}
By \eqref{eqn22}, we know that
\begin{eqnarray}\label{eqn25}
\!\!\!\!\!\!\!\!&&\Bbb{E}\Big[\sup_{s\in[0,t]}\|X^\varepsilon_\lambda(s)-X^\varepsilon_{\lambda'}(s)\|^2_{F^*_{1,2}}\Big]+4\tilde{\alpha}\Bbb{E}\int_0^t\big|\Psi(X^\varepsilon_\lambda(s))-\Psi(X^\varepsilon_{\lambda'}(s))\big|_2^2ds\nonumber\\
\leq\!\!\!\!\!\!\!\!&&e^{\frac{C}{\sqrt{\varepsilon}}}\cdot 16e^{C_1T}(|x|_2^2+C_2)(\lambda+\lambda').
\end{eqnarray}
\eqref{eqn25} implies that there exists an $F^*_{1,2}$-valued c\`{a}dl\`{a}g $\mathcal{F}_t$-adapted process $\{X^\varepsilon(t)\}_{t\in[0,T]}$ such that $X^\varepsilon\in L^2(\Omega;L^\infty([0,T];F^*_{1,2}))$ and $X^\varepsilon\in D([0,T];F^*_{1,2})$, $\Bbb{P}$-a.s.. This together with Claim \ref{claim1} imply that
$$X^\varepsilon\in L^2([0,T]\times \Omega;L^2(\mu)).$$\hspace{\fill}$\Box$
\end{proof}

\begin{claim}\label{claim3}
$X^\varepsilon$ satisfies \eqref{eq:3} and $\int_0^\cdot\Psi(X^\varepsilon(s))ds\in C([0,T];F^*_{1,2})$, $\Bbb{P}$-a.s..
\end{claim}
\begin{proof}
First, let us clarify that $X^\varepsilon$ satisfies \eqref{eq:3}. From Claim \ref{claim2}, we know that as $\lambda\rightarrow0$,
\begin{eqnarray}\label{eqn26}
X^\varepsilon_\lambda\rightarrow X^\varepsilon\ \text{in}\ L^2(\Omega;L^\infty([0,T];F^*_{1,2})).
\end{eqnarray}
By BDG's inequality, \textbf{(H3)} and \eqref{eqn26}, we have
\begin{eqnarray*}
\!\!\!\!\!\!\!\!&&\Bbb{E}\Bigg[\sup_{t\in[0,T]}\Big\|\int_0^t\int_{Z}f(s,X^\varepsilon_\lambda(s-),z)-f(s,X^\varepsilon(s-),z)\widetilde{N}(ds,dz)\Big\|^2_{F^*_{1,2}}\Bigg]\\
\!\!\!\!\!\!\!\!&&\leq C\Bbb{E}\Bigg[\int_0^T\int_{Z}\big\|f(s,X^\varepsilon_\lambda(s),z)-f(s,X^\varepsilon(s),z)\big\|^2_{F^*_{1,2}}\nu(dz)ds\Bigg]\\
\!\!\!\!\!\!\!\!&&\leq C\Bbb{E}\Big[\int_0^T\big\|X^\varepsilon_\lambda(s)-X^\varepsilon(s)\big\|^2_{F^*_{1,2}}ds\Big]\\
\!\!\!\!\!\!\!\!&&\leq CT\Bbb{E}\Big[\sup_{s\in[0,T]}\big\|X^\varepsilon_\lambda(s)-X^\varepsilon(s)\big\|^2_{F^*_{1,2}}\Big]\\
\!\!\!\!\!\!\!\!&&\longrightarrow0,\ \ \text{as}\ \ \lambda\longrightarrow0,
\end{eqnarray*}
which means that as $\lambda\rightarrow0$,
\begin{eqnarray}\label{eqn27}
\!\!\!\!\!\!\!\!&&\int_0^\cdot\int_{Z}f(s,X^\varepsilon_\lambda(s-),z)\widetilde{N}(ds,dz)\nonumber\\
\!\!\!\!\!\!\!\!&&\longrightarrow \int_0^\cdot\int_{Z}f(s,X^\varepsilon(s-),z)\widetilde{N}(ds,dz)\ \text{in}\ L^2(\Omega;L^\infty([0,T];F^*_{1,2})).
\end{eqnarray}
Notice that
\begin{eqnarray*}
\!\!\!\!\!\!\!\!&&\Bbb{E}\int_0^T\big|\Psi(X^\varepsilon_\lambda(s))+\lambda X^\varepsilon_\lambda(s)\big|_2^2ds\nonumber\\
\leq\!\!\!\!\!\!\!\!&&\Bbb{E}\int_0^T\big|\Psi(X^\varepsilon_\lambda(s))\big|_2^2+\lambda^2\big|X^\varepsilon_\lambda(s)\big|_2^2ds\nonumber\\
\leq\!\!\!\!\!\!\!\!&&\big((Lip\Psi)^2+\lambda^2\big)\Bbb{E}\int_0^T\big|X^\varepsilon_\lambda(s)\big|_2^2ds,
\end{eqnarray*}
which indicates
\begin{eqnarray}\label{eqn31}
\Psi(X^\varepsilon_\lambda(\cdot))+\lambda X^\varepsilon_\lambda(\cdot)\ \text{converges\ weakly\  to\ some\ element}\ Y \ \text{in}\ L^2(\Omega;L^2([0,T];L^2(\mu))).
\end{eqnarray}
Recall that $\forall t\in[0,T]$,
\begin{eqnarray}\label{eqn28}
X^\varepsilon_\lambda(t)=x+\int_0^t(L-\varepsilon)(\Psi(X^\varepsilon_\lambda(s))+\lambda X^\varepsilon_\lambda(s))ds+\int_0^t\int_{Z}f(s,X^\varepsilon_\lambda(s-),z)\widetilde{N}(ds,dz),
\end{eqnarray}
holds in $(L^2(\mu))^*$. Notice that from \eqref{eqn26}-\eqref{eqn28}, we know $\forall t\in[0,T]$,
\begin{eqnarray}\label{Eq Zhai 1}
X^\varepsilon(t)=x+\int_0^t(L-\varepsilon)Y(s)ds+\int_0^t\int_{Z}f(s,X^\varepsilon(s-),z)\widetilde{N}(ds,dz)\ \text{holds\ in}\ (L^2(\mu))^*.
\end{eqnarray}
So, in order to prove that $X^\varepsilon$ satisfies \eqref{eq:3},
it remains to show $Y(\cdot)=\Psi(X^\varepsilon(\cdot))$, $dt\otimes\Bbb{P}$-a.s..

\skip 0.2cm
Now, applying It\^{o}'s formula to $\|X^\varepsilon(t)\|^2_{F^*_{1,2,\varepsilon}}$ in $F^*_{1,2}$, we get
\begin{eqnarray}\label{eqn32}
\|X^\varepsilon(t)\|^2_{F^*_{1,2,\varepsilon}}=\!\!\!\!\!\!\!\!&&\|x\|^2_{F^*_{1,2,\varepsilon}}-2\int_0^t\big\langle Y(s),X^\varepsilon(s)\big\rangle_2ds\nonumber\\
\!\!\!\!\!\!\!\!&&+2\int_0^t\int_{Z}\big\langle X^\varepsilon(s-),f(s,X^\varepsilon(s-),z)\big\rangle_{F^*_{1,2,\varepsilon}}\widetilde{N}(ds,dz)\nonumber\\
\!\!\!\!\!\!\!\!&&+\int_0^t\int_{Z}\|f(s,X^\varepsilon(s-),z)\|^2_{F^*_{1,2,\varepsilon}}N(ds,dz).
\end{eqnarray}
Applying It\^{o}'s formula to the process $X^\varepsilon_\lambda$, see \cite[P304]{BLZ}, we have
\begin{eqnarray*}
\!\!\!\!\!\!\!\!&&e^{-Kt}\|X^\varepsilon_\lambda(t)\|^2_{F^*_{1,2,\varepsilon}}\nonumber\\
=\!\!\!\!\!\!\!\!&&\|x\|^2_{F^*_{1,2,\varepsilon}}+2\int_0^t\int_{Z}e^{-Ks}\big\langle X^\varepsilon_\lambda(s-),f(s,X^\varepsilon_\lambda(s-),z)\big\rangle_{F^*_{1,2,\varepsilon}}\widetilde{N}(ds,dz)\nonumber\\
\!\!\!\!\!\!\!\!&&+\int_0^t\int_{Z}e^{-Ks}\|f(s,X^\varepsilon_\lambda(s-),z)\|^2_{F^*_{1,2,\varepsilon}}N(ds,dz)\nonumber\\
\!\!\!\!\!\!\!\!&&+\int_0^te^{-Ks}\cdot\Big(2_{(L^2(\mu))^*}\big\langle(L-\varepsilon)(\Psi(X^\varepsilon_\lambda(s))+\lambda X^\varepsilon_\lambda(s)),X^\varepsilon_\lambda(s)\big\rangle_{L^2(\mu)}-K\|X^\varepsilon_\lambda(s)\|^2_{F^*_{1,2,\varepsilon}}\Big)ds.
\end{eqnarray*}
Taking expectation of both sides to the above equality and by \textbf{(H1)}, we get for $\phi\in L^\infty([0,T];L^2(\Omega;F^*_{1,2}))\cap L^2([0,T]\times\Omega;\mathcal{B}\mathcal{F},dt\otimes\Bbb{P};L^2(\mu))$,
\begin{eqnarray*}
\!\!\!\!\!\!\!\!&&\Bbb{E}\Big[e^{-Kt}\|X^\varepsilon_\lambda(t)\|^2_{F^*_{1,2,\varepsilon}}\Big]-\Bbb{E}\|x\|^2_{F^*_{1,2,\varepsilon}}\nonumber\\
\leq\!\!\!\!\!\!\!\!&&\Bbb{E}\!\Bigg[\!\int_0^t\!\!e^{-Ks}\Big(2_{(L^2(\mu))^*}\big\langle(L-\varepsilon)\big(\Psi(X^\varepsilon_\lambda(s))+\lambda X^\varepsilon_\lambda(s)\big)-(L-\varepsilon)\big(\Psi(\phi(s))+\lambda \phi(s)\big),X^\varepsilon_\lambda(s)-\phi(s)\big\rangle_{L^2(\mu)}\nonumber\\
\!\!\!\!\!\!\!\!&&~~~~-K\|X^\varepsilon_\lambda(s)-\phi(s)\|^2_{F^*_{1,2,\varepsilon}}+\int_{Z}\|f(s,X^\varepsilon_\lambda(s),z)-f(s,\phi(s),z)\|^2_{F^*_{1,2,\varepsilon}}\nu(dz)\Big)ds\Bigg]\nonumber\\
\!\!\!\!\!\!\!\!&&+\Bbb{E}\!\Bigg\{\!\int_0^t\!\!e^{-Ks}\Bigg(2_{(L^2(\mu))^*}\big\langle(L-\varepsilon)\big(\Psi(X^\varepsilon_\lambda(s))+\lambda X^\varepsilon_\lambda(s)\big)-(L-\varepsilon)\big(\Psi(\phi(s))+\lambda \phi(s)\big),\phi(s)\big\rangle_{L^2(\mu)}\nonumber\\
\!\!\!\!\!\!\!\!&&~~~~+2_{(L^2(\mu))^*}\big\langle(L-\varepsilon)(\phi(s))+\lambda \phi(s)),X^\varepsilon_\lambda(s)\big\rangle_{L^2(\mu)}-2K\big\langle X^\varepsilon_\lambda(s),\phi(s)\big\rangle_{F^*_{1,2,\varepsilon}}+K\|\phi(s)\|^2_{F^*_{1,2,\varepsilon}}\nonumber\\
\!\!\!\!\!\!\!\!&&~~~~+\int_{Z}\Big(2\big\langle f(s,X^\varepsilon_\lambda(s),z),f(s,\phi(s),z)\big\rangle_{F^*_{1,2,\varepsilon}}-\|f(s,\phi(s),z)\|^2_{F^*_{1,2,\varepsilon}}\Big)\nu(dz)\Bigg)ds\Bigg\}.
\end{eqnarray*}
Choosing $K$ to be $\frac{2(1-\varepsilon)^2}{(Lip\Psi+1)^{-1}}+C_3$ as in \eqref{eqn8}. After some simple rearrangement, we find
\begin{eqnarray*}
\!\!\!\!\!\!\!\!&&\Bbb{E}\Big[e^{-Kt}\|X^\varepsilon_\lambda(t)\|^2_{F^*_{1,2,\varepsilon}}\Big]-\Bbb{E}\|x\|^2_{F^*_{1,2,\varepsilon}}\nonumber\\
\leq\!\!\!\!\!\!\!\!&&\Bbb{E}\!\Bigg\{\!\int_0^t\!\!e^{-Ks}\Bigg(2_{(L^2(\mu))^*}\big\langle(L-\varepsilon)\big(\Psi(X^\varepsilon_\lambda(s))+\lambda X^\varepsilon_\lambda(s)\big)-(L-\varepsilon)\big(\Psi(\phi(s))+\lambda \phi(s)\big),\phi(s)\big\rangle_{L^2(\mu)}\nonumber\\
\!\!\!\!\!\!\!\!&&~~~~+2_{(L^2(\mu))^*}\big\langle(L-\varepsilon)\big(\phi(s))+\lambda \phi(s)\big),X^\varepsilon_\lambda(s)\big\rangle_{L^2(\mu)}-2K\big\langle X^\varepsilon_\lambda(s),\phi(s)\big\rangle_{F^*_{1,2,\varepsilon}}+K\|\phi(s)\|^2_{F^*_{1,2,\varepsilon}}\nonumber\\
\!\!\!\!\!\!\!\!&&~~~~+\int_{Z}\Big(2\big\langle f(s,X^\varepsilon_\lambda(s),z),f(s,\phi(s),z)\big\rangle_{F^*_{1,2,\varepsilon}}-\|f(s,\phi(s),z)\|^2_{F^*_{1,2,\varepsilon}}\Big)\nu(dz)\Bigg)ds\Bigg\}.
\end{eqnarray*}
This together with \eqref{eqn26}, \eqref{eqn27} and \eqref{eqn31} gives for any nonnegative function $\psi\in L^\infty([0,T];dt)$ that
\begin{eqnarray*}
\!\!\!\!\!\!\!\!&&\Bbb{E}\Big[\int_0^T\psi(t)\big(e^{-Kt}\|X^\varepsilon(t)\|^2_{F^*_{1,2,\varepsilon}}-\|x\|^2_{F^*_{1,2,\varepsilon}}\big)dt\Big]\nonumber\\
\leq\!\!\!\!\!\!\!\!&&\lim_{\lambda\rightarrow0}\inf\Bbb{E}\Big[\int_0^T\psi(t)\big(e^{-Kt}\|X^\varepsilon_\lambda(t)\|^2_{F^*_{1,2,\varepsilon}}-\|x\|^2_{F^*_{1,2,\varepsilon}}\big)dt\Big]\nonumber\\
\leq\!\!\!\!\!\!\!\!&&\lim_{\lambda\rightarrow0}\inf\Bbb{E}\Bigg\{\int_0^T\psi(t)\Bigg(\int_0^te^{-Ks}\Big(2_{(L^2(\mu))^*}\big\langle(L-\varepsilon)\big(\Psi(X^\varepsilon_\lambda(s))+\lambda X^\varepsilon_\lambda(s)\big)\nonumber\\
\!\!\!\!\!\!\!\!&&~~~~~~~~~~~~~~~~~~~~~~~~~~~~~~~~~~~~~~~~~~~~~~~~~~~~~~~-(L-\varepsilon)\big(\Psi(\phi(s))+\lambda \phi(s)\big),\phi(s)\big\rangle_{L^2(\mu)}\nonumber\\
\!\!\!\!\!\!\!\!&&~~~~+2_{(L^2(\mu))^*}\big\langle(L-\varepsilon)\big(\Psi(\phi(s))+\lambda \phi(s)\big),X^\varepsilon_\lambda(s)\big\rangle_{L^2(\mu)}-2K\langle X^\varepsilon_\lambda(s),\phi(s)\rangle_{F^*_{1,2,\varepsilon}}+K\|\phi(s)\|^2_{F^*_{1,2,\varepsilon}}\nonumber\\
\!\!\!\!\!\!\!\!&&~~~~+\int_{Z}\big(2\langle f(s,X^\varepsilon_\lambda(s),z),f(s,\phi(s),z)\rangle_{F^*_{1,2,\varepsilon}}-\|f(s,\phi(s),z)\|^2_{F^*_{1,2,\varepsilon}}\big)\nu(dz)\Big)ds\Bigg)dt\Bigg\}.
\end{eqnarray*}
Again by \eqref{eqn26}, \eqref{eqn27} and \eqref{eqn31}, we infer
\begin{eqnarray}\label{eqn33}
\!\!\!\!\!\!\!\!&&\Bbb{E}\Big[\int_0^T\psi(t)\big(e^{-Kt}\|X^\varepsilon(t)\|^2_{F^*_{1,2,\varepsilon}}-\|x\|^2_{F^*_{1,2,\varepsilon}}\big)dt\Big]\nonumber\\
\leq\!\!\!\!\!\!\!\!&&\Bbb{E}\Bigg\{\int_0^T\psi(t)\Bigg(\int_0^te^{-Ks}\Big(2_{(L^2(\mu))^*}\big\langle(L-\varepsilon)Y(s)-(L-\varepsilon)(\Psi(\phi(s)),\phi(s)\big\rangle_{L^2(\mu)}\nonumber\\
\!\!\!\!\!\!\!\!&&~~~~+2_{(L^2(\mu))^*}\big\langle(L-\varepsilon)(\Psi(\phi(s)),X^\varepsilon(s)\big\rangle_{L^2(\mu)}-2K\langle X^\varepsilon(s),\phi(s)\rangle_{F^*_{1,2,\varepsilon}}+K\|\phi(s)\|^2_{F^*_{1,2,\varepsilon}}\nonumber\\
\!\!\!\!\!\!\!\!&&~~~~+\int_{Z}\!\!\big(2\langle f(s,X^\varepsilon(s),z),f(s,\phi(s),z)\rangle_{F^*_{1,2,\varepsilon}}\!-\!\|f(s,\phi(s),z)\|^2_{F^*_{1,2,\varepsilon}}\big)\nu(dz)\Big)ds\!\!\Bigg)dt\Bigg\}.
\end{eqnarray}
On the other hand, by \eqref{eqn32} we infer that
\begin{eqnarray}\label{eqn34}
\!\!\!\!\!\!\!\!&&\Bbb{E}\Big[\int_0^te^{-Ks}\|X^\varepsilon(s)\|^2_{F^*_{1,2,\varepsilon}}-\|x\|^2_{F^*_{1,2,\varepsilon}}ds\Big]\nonumber\\
=\!\!\!\!\!\!\!\!&&\Bbb{E}\Bigg[\int_0^te^{-Ks}\Big(2_{(L^2(\mu))^*}\big\langle(L-\varepsilon)Y(s),X^\varepsilon(s)\big\rangle_{L^2(\mu)}-K\|X^\varepsilon(s)\|^2_{F^*_{1,2,\varepsilon}}\nonumber\\
\!\!\!\!\!\!\!\!&&~~~~~~~~~~~~~~~~~~~~~~~~~~~+\int_{Z}\|f(s,X^\varepsilon(s),z)\|^2_{F^*_{1,2,\varepsilon}}\nu(dz)\Big)ds\Bigg].
\end{eqnarray}
Combining \eqref{eqn34} with \eqref{eqn33}, we have
\begin{eqnarray}\label{eqn35}
\!\!\!\!\!\!\!\!&&\Bbb{E}\Bigg[\int_0^T\psi(t)\Big(\int_0^te^{-Ks}\big(2_{(L^2(\mu))^*}\big\langle(L-\varepsilon)Y(s)-(L-\varepsilon)(\Psi(\phi(s))),X^\varepsilon(s)-\phi(s)\big\rangle_{L^2(\mu)}\nonumber\\
\!\!\!\!\!\!\!\!&&~~~~~~-K\|X^\varepsilon(s)-\phi(s)\|^2_{F^*_{1,2,\varepsilon}}+\int_{Z}\|f(s,\phi(s),z)-f(s,X^\varepsilon(s),z)\|^2_{F^*_{1,2,\varepsilon}}\nu(dz)\big)ds\Big)dt\Bigg]\nonumber\\
\leq\!\!\!\!\!\!\!\!&&0.
\end{eqnarray}
Note that \eqref{eqn35} also implies
\begin{eqnarray}\label{eqn36}
\!\!\!\!\!\!\!\!&&\Bbb{E}\Bigg[\int_0^T\psi(t)\Big(\int_0^te^{-Ks}\big(2_{(L^2(\mu))^*}\big\langle(L-\varepsilon)Y(s)-(L-\varepsilon)(\Psi(\phi(s))),X^\varepsilon(s)-\phi(s)\big\rangle_{L^2(\mu)}\nonumber\\
\!\!\!\!\!\!\!\!&&~~~~~~-K\|X^\varepsilon(s)-\phi(s)\|^2_{F^*_{1,2,\varepsilon}}\big)ds\Big)dt\Bigg]\leq0.
\end{eqnarray}
Put $\phi=X^\varepsilon-\epsilon\tilde{\phi}u$ in \eqref{eqn36} for $\tilde{\phi}\in L^\infty([0,T]\times\Omega;dt\otimes\Bbb{P};\Bbb{R})$ and $u\in L^2(\mu)$, divide both sides by $\epsilon$ and let $\epsilon\rightarrow0$. Then we have
\begin{eqnarray*}
\!\!\!\!\!\!\!\!&&\Bbb{E}\Bigg[\int_0^T\psi(t)\Big(\int_0^te^{-Ks}\big(2_{(L^2(\mu))^*}\big\langle(L-\varepsilon)Y(s)-(L-\varepsilon)(\Psi(X^\varepsilon(s))),u\big\rangle_{L^2(\mu)}\big)ds\Big)dt\Bigg]\leq0.
\end{eqnarray*}
Hence, we infer
\begin{eqnarray}\label{eqn35.1}
Y(\cdot)=\Psi(X^\varepsilon(\cdot)), dt\otimes\Bbb{P}\text{-}a.s..
\end{eqnarray}

Next, let us prove that $\int_0^\cdot\Psi(X^\varepsilon(s))ds\in C([0,T];F_{1,2})$, $\Bbb{P}$-a.s.. On the one hand, from \eqref{eqn28} and \cite[Remark 1.1]{BLZ}, we know that
$$\int_0^\cdot(L-\varepsilon)(\Psi(X^\varepsilon_\lambda(s))+\lambda X^\varepsilon_\lambda(s))ds\in C([0,T];(L^2(\mu))^*),\ \Bbb{P}\text{-a.s}.,$$
$$X^\varepsilon_\lambda\in D([0,T];(L^2(\mu))^*),\ \Bbb{P}\text{-a.s}.,$$
$$\int_0^\cdot\int_{Z}f(s,X^\varepsilon_\lambda(s-),z)\widetilde{N}(ds,dz)\in D([0,T];(L^2(\mu))^*),\ \Bbb{P}\text{-a.s}.,$$
on the other hand, from Claim \ref{claim2}, we know that $X^\varepsilon_\lambda\in D([0,T];F^*_{1,2})$, $\Bbb{P}$-a.s., and
$$\int_0^\cdot\int_{Z}f(s,X^\varepsilon_\lambda(s-),z)\widetilde{N}(ds,dz)\in D([0,T];F^*_{1,2}),\ \Bbb{P}\text{-a.s}..$$
So
$$\int_0^\cdot(L-\varepsilon)(\Psi(X^\varepsilon_\lambda(s))+\lambda X^\varepsilon_\lambda(s))ds\in C([0,T];F^*_{1,2}),\ \Bbb{P}\text{-a.s}.,$$
apparently,
\begin{eqnarray}\label{eqn29}
\int_0^\cdot\Psi(X^\varepsilon_\lambda(s))+\lambda X^\varepsilon_\lambda(s)ds\in C([0,T];F_{1,2}),\ \Bbb{P}\text{-a.s.}.
\end{eqnarray}
Taking \eqref{eqn26}-\eqref{eqn28}, \eqref{eqn35.1} and \eqref{eqn29} into account, we know that as $\lambda\rightarrow0$,
\begin{eqnarray}\label{eqn290}
\int_0^\cdot\Psi(X^\varepsilon_\lambda(s))+\lambda X^\varepsilon_\lambda(s)ds\rightarrow\int_0^\cdot\Psi(X^\varepsilon(s))ds\ \text{in}\ L^2(\Omega;C([0,T];F_{1,2})),
\end{eqnarray}
which indicates $\int_0^\cdot\Psi(X^\varepsilon(s))ds\in C([0,T];F_{1,2})$, $\Bbb{P}$-a.s..

The proof of Claim \ref{claim3} is complete. \hspace{\fill}$\Box$
\end{proof}

\vspace{2mm}
\textbf{Uniqueness}
\vspace{2mm}

If $X^\varepsilon_1$, $X^\varepsilon_2$ are two solutions to \eqref{eq:3}, we have $\Bbb{P}$-a.s.,
\begin{eqnarray*}
\!\!\!\!\!\!\!\!&&X^\varepsilon_1(t)-X^\varepsilon_2(t)+(\varepsilon-L)\int_0^t\Psi(X^\varepsilon_1(s))-\Psi(X^\varepsilon_2(s))ds\nonumber\\
=\!\!\!\!\!\!\!\!&&\int_0^t\int_{Z}f(s,X^\varepsilon_1(s-),z)-f(s,X^\varepsilon_2(s-),z)\widetilde{N}(ds,dz),\ \forall t\in[0,T].
\end{eqnarray*}
Applying It\^{o}'s formula to $\|X^\varepsilon_1(t)-X^\varepsilon_2(t)\|^2_{F^*_{1,2,\varepsilon}}$ in $F^*_{1,2}$, we get
\begin{eqnarray}\label{eqn37}
\!\!\!\!\!\!\!\!&&\|X^\varepsilon_1(t)-X^\varepsilon_2(t)\|^2_{F^*_{1,2,\varepsilon}}+2\int_0^t\big\langle\Psi(X^\varepsilon_1(s))-\Psi(X^\varepsilon_2(s)),X^\varepsilon_1(s)-X^\varepsilon_2(s)\big\rangle_2ds\nonumber\\
=\!\!\!\!\!\!\!\!&&\int_0^t\int_{Z}\|f(s,X^\varepsilon_1(s-),z)-f(s,X^\varepsilon_2(s-),z)\|^2_{F^*_{1,2,\varepsilon}}N(ds,dz)\nonumber\\
\!\!\!\!\!\!\!\!&&+2\int_0^t\!\int_{Z}\big\langle X^\varepsilon_1(s-)-X^\varepsilon_2(s-),f(s,X^\varepsilon_1(s-),z)-f(s,X^\varepsilon_2(s-),z)\big\rangle_{F^*_{1,2,\varepsilon}}\!\!\widetilde{N}(ds,dz).
\end{eqnarray}
Since $\Psi$ is Lipschitz, we have
\begin{eqnarray}\label{eqn38}
\big(\Psi(r)-\Psi(r')\big)(r-r')\geq (Lip\Psi+1)^{-1}|\Psi(r)-\Psi(r')|^2,\ \forall r, r'\in\Bbb{R}.
\end{eqnarray}
Taking expectation of both sides to \eqref{eqn37}, then taking \eqref{eqn38} and \textbf{(H3)} into account, we obtain
\begin{eqnarray}\label{eqn381}
\!\!\!\!\!\!\!\!&&\Bbb{E}\|X^\varepsilon_1(t)-X^\varepsilon_2(t)\|^2_{F^*_{1,2,\varepsilon}}+2(Lip\Psi+1)^{-1}\Bbb{E}\int_0^t|\Psi(X^\varepsilon_1(s))-\Psi(X^\varepsilon_2(s))|_2^2ds\nonumber\\
\leq\!\!\!\!\!\!\!\!&&C_3\int_0^t\Bbb{E}\|X^\varepsilon_1(s)-X^\varepsilon_2(s)\|^2_{F^*_{1,2,\varepsilon}}ds.
\end{eqnarray}
The second term in the left-hand side of the above inequality is positive, thus we have
\begin{eqnarray*}
\Bbb{E}\|X^\varepsilon_1(t)-X^\varepsilon_2(t)\|^2_{F^*_{1,2,\varepsilon}}\leq C_3\int_0^t\Bbb{E}\|X^\varepsilon_1(s)-X^\varepsilon_2(s)\|^2_{F^*_{1,2,\varepsilon}}ds.
\end{eqnarray*}
By Gronwall's inequality, we get $X^\varepsilon_1(t)=X^\varepsilon_2(t)$, $\Bbb{P}$-a.s., $\forall t\in[0,T]$, which indicates the uniqueness.
Hence the proof of Proposition \ref{Th2} is complete.\hspace{\fill}$\Box$
\end{proof}
\vspace{2mm}
\smallskip

\section{Proof of Theorem \ref{Th1}}
Based on Proposition \ref{Th2}, we are now ready to prove our main result Theorem \ref{Th1}. The idea is to prove the sequence $\{X^\varepsilon\}_{\varepsilon\in(0,1)}$ converges to the solution of \eqref{eq:1} as $\varepsilon\rightarrow0$.
\setcounter{equation}{0}
 \setcounter{definition}{0}

\begin{proof}
First, we rewrite \eqref{eq:3} as following:
\begin{eqnarray*}
\!\!\!\!\!\!\!\!&&X^\varepsilon(t)+(1-L)\int_0^t\Psi(X^\varepsilon(s))ds\nonumber\\
=\!\!\!\!\!\!\!\!&&x+(1-\varepsilon)\int_0^t\Psi(X^\varepsilon(s))ds+\int_0^t\int_{Z}f(s,X^\varepsilon(s-),z)\widetilde{N}(ds,dz).
\end{eqnarray*}
Apply It\^{o}'s formula to $\|X^\varepsilon(t)\|^2_{F^*_{1,2}}$, and after taking expectation to both sides, we have
\begin{eqnarray}\label{eqn39}
\!\!\!\!\!\!\!\!&&\Bbb{E}\|X^\varepsilon(t)\|^2_{F^*_{1,2}}+2\Bbb{E}\int_0^t\langle\Psi(X^\varepsilon(s)),X^\varepsilon(s)\rangle_2ds\nonumber\\
=\!\!\!\!\!\!\!\!&&\Bbb{E}\|x\|^2_{F^*_{1,2}}+2(1-\varepsilon)\Bbb{E}\int_0^t\langle\Psi(X^\varepsilon(s)),X^\varepsilon(s)\rangle_{F^*_{1,2}}ds\nonumber\\
\!\!\!\!\!\!\!\!&&+\Bbb{E}\int_0^t\int_{Z}\|f(s,X^\varepsilon(s-),z)\|^2_{F^*_{1,2}}N(ds,dz).
\end{eqnarray}
Since $\Psi$ is Lipschitz, we have
\begin{eqnarray}\label{eqn40}
\Psi(r)r\geq\tilde{\alpha}|\Psi(r)|^2,\ \forall r\in\Bbb{R}.
\end{eqnarray}
By \eqref{eqn39}, \eqref{eqn40} and \textbf{(H2)}, we have
\begin{eqnarray}\label{eqn41}
\!\!\!\!\!\!\!\!&&\Bbb{E}\|X^\varepsilon(t)\|^2_{F^*_{1,2}}+2\tilde{\alpha}\Bbb{E}\int_0^t|\Psi(X^\varepsilon(s))|_2^2ds\nonumber\\
\leq\!\!\!\!\!\!\!\!&&\Bbb{E}\|x\|^2_{F^*_{1,2}}+2\Bbb{E}\int_0^t\|\Psi(X^\varepsilon(s))\|_{F^*_{1,2}}\cdot\|X^\varepsilon(s)\|_{F^*_{1,2}}ds\nonumber\\
\!\!\!\!\!\!\!\!&&+C_1\Bbb{E}\int_0^t\|X^\varepsilon(s)\|^2_{F^*_{1,2}}ds+C_1.
\end{eqnarray}
Since $L^2(\mu)$ is continuously embedded into $F^*_{1,2}$, and by Young's inequality, we know that
\begin{eqnarray}\label{eqn42}
\!\!\!\!\!\!\!\!&&\Bbb{E}\int_0^t\|\Psi(X^\varepsilon(s))\|_{F^*_{1,2}}\cdot\|X^\varepsilon(s)\|_{F^*_{1,2}}ds\nonumber\\
\leq\!\!\!\!\!\!\!\!&&\Bbb{E}\int_0^t|\Psi(X^\varepsilon(s))|_2\cdot\|X^\varepsilon(s)\|_{F^*_{1,2}}ds\nonumber\\
\leq\!\!\!\!\!\!\!\!&&\tilde{\alpha}\Bbb{E}\int_0^t|\Psi(X^\varepsilon(s))|_2^2ds+\frac{1}{4\tilde{\alpha}}\Bbb{E}\int_0^t\|X^\varepsilon(s)\|^2_{F^*_{1,2}}ds.
\end{eqnarray}
Taking \eqref{eqn42} into \eqref{eqn41}, after some simple rearrangements, we get that
\begin{eqnarray*}
\!\!\!\!\!\!\!\!&&\Bbb{E}\|X^\varepsilon(t)\|^2_{F^*_{1,2}}+\tilde{\alpha}\Bbb{E}\int_0^t|\Psi(X^\varepsilon(s))|_2^2ds\nonumber\\
\leq\!\!\!\!\!\!\!\!&&\Bbb{E}\|x\|^2_{F^*_{1,2}}+(\frac{1}{2\tilde{\alpha}}+C_1)\Bbb{E}\int_0^t\|X^\varepsilon(s)\|^2_{F^*_{1,2}}ds+C_1.
\end{eqnarray*}
By Gronwall's inequality, we know that
\begin{eqnarray}\label{eqn43}
\Bbb{E}\|X^\varepsilon(t)\|^2_{F^*_{1,2}}+\tilde{\alpha}\Bbb{E}\int_0^t|\Psi(X^\varepsilon(s))|_2^2ds\leq\big(\|x\|^2_{F^*_{1,2}}+C_1\big)\cdot e^{(\frac{1}{2\tilde{\alpha}}+C_1)T}.
\end{eqnarray}
 In the following, we will prove the convergence of $\{X^\varepsilon\}_{\varepsilon\in(0,1)}$.

 Apply It\^{o}'s formula to $\|X^\varepsilon(t)-X^{\varepsilon'}(t)\|^2_{F^*_{1,2}}$, $\varepsilon, \varepsilon'\in(0,1)$, we get, for all $t\in[0,T]$,
 \begin{eqnarray}\label{eqn44}
\!\!\!\!\!\!\!\!&&\|X^\varepsilon(t)-X^{\varepsilon'}(t)\|^2_{F^*_{1,2}}+2\int_0^t\big\langle\Psi(X^\varepsilon(s))-\Psi(X^{\varepsilon'}(s)),X^\varepsilon(s)-X^{\varepsilon'}(s)\big\rangle_2ds\nonumber\\
=\!\!\!\!\!\!\!\!&&2\int_0^t\big\langle\Psi(X^\varepsilon(s))-\Psi(X^{\varepsilon'}(s)),X^\varepsilon(s)-X^{\varepsilon'}(s)\big\rangle_{F^*_{1,2}}ds\nonumber\\
\!\!\!\!\!\!\!\!&&-2\int_0^t\big\langle\varepsilon\Psi(X^\varepsilon(s))-\varepsilon'\Psi(X^{\varepsilon'}(s)),X^\varepsilon(s)-X^{\varepsilon'}(s)\big\rangle_{F^*_{1,2}}ds\nonumber\\
\!\!\!\!\!\!\!\!&&+\int_0^t\int_{Z}\|f(s,X^\varepsilon(s-),z)-f(s,X^{\varepsilon'}(s-),z)\|^2_{F^*_{1,2}}N(ds,dz)\nonumber\\
\!\!\!\!\!\!\!\!&&+2\int_0^t\int_{Z}\big\langle X^\varepsilon(s-)-X^{\varepsilon'}(s-),f(s,X^\varepsilon(s-),z)-f(s,X^{\varepsilon'}(s-),z)\big\rangle_{F^*_{1,2}}\widetilde{N}(ds,dz).
\end{eqnarray}
Since $L^2(\mu)$ continuously embedded into $F^*_{1,2}$, the second term in the right-hand side of \eqref{eqn44} can be dominated by
\begin{eqnarray}\label{eqn45}
\!\!\!\!\!\!\!\!&&-2\int_0^t\big\langle\varepsilon\Psi(X^\varepsilon(s))-\varepsilon'\Psi(X^{\varepsilon'}(s)),X^\varepsilon(s)-X^{\varepsilon'}(s)\big\rangle_{F^*_{1,2}}ds\nonumber\\
\leq\!\!\!\!\!\!\!\!&&2C\int_0^t\big(\varepsilon|\Psi(X^\varepsilon(s))|_2+\varepsilon'|\Psi(X^{\varepsilon'}(s))|_2\big)\cdot\|X^\varepsilon(s)-X^{\varepsilon'}(s)\|^2_{F^*_{1,2}}ds.
\end{eqnarray}
From \cite[(3.42)]{RWX}, we know that
\begin{eqnarray}\label{eqn46}
\!\!\!\!\!\!\!\!&&2\int_0^t\big\langle\Psi(X^\varepsilon(s))-\Psi(X^{\varepsilon'}(s)),X^\varepsilon(s)-X^{\varepsilon'}(s)\big\rangle_2ds\nonumber\\
\geq\!\!\!\!\!\!\!\!&&2\tilde{\alpha}\int_0^t|\Psi(X^\varepsilon(s))-\Psi(X^{\varepsilon'}(s))|_2^2ds.
\end{eqnarray}
Taking \eqref{eqn45} and \eqref{eqn46} into \eqref{eqn44}, we get that
\begin{eqnarray}\label{eqn47}
\!\!\!\!\!\!\!\!&&\|X^\varepsilon(t)-X^{\varepsilon'}(t)\|^2_{F^*_{1,2}}+2\tilde{\alpha}\int_0^t|\Psi(X^\varepsilon(s))-\Psi(X^{\varepsilon'}(s))|_2^2ds\nonumber\\
\leq\!\!\!\!\!\!\!\!&&C_1\int_0^t|\Psi(X^\varepsilon(s))-\Psi(X^{\varepsilon'}(s))|_2\cdot\|X^\varepsilon(s)-X^{\varepsilon'}(s)\|_{F^*_{1,2}}ds\nonumber\\
\!\!\!\!\!\!\!\!&&+C_2\int_0^t\big(\varepsilon|\Psi(X^\varepsilon(s))|_2+\varepsilon'|\Psi(X^{\varepsilon'}(s))|_2\big)\cdot\|X^\varepsilon(s)-X^{\varepsilon'}(s)\|_{F^*_{1,2}}ds\nonumber\\
\!\!\!\!\!\!\!\!&&+\int_0^t\int_{Z}\|f(s,X^\varepsilon(s-),z)-f(s,X^{\varepsilon'}(s-),z)\|^2_{F^*_{1,2}}N(ds,dz)\nonumber\\
\!\!\!\!\!\!\!\!&&+2\int_0^t\int_{Z}\big\langle X^\varepsilon(s-)-X^{\varepsilon'}(s-),f(s,X^\varepsilon(s-),z)-f(s,X^{\varepsilon'}(s-),z)\big\rangle_{F^*_{1,2}}\widetilde{N}(ds,dz).
\end{eqnarray}
Taking expectation to both sides of \eqref{eqn47}, by Young's equality, BDG's inequality and \textbf{(H3)}, we obtain that, for all $t\in[0,T]$,
\begin{eqnarray*}
&&\Bbb{E}\Big[\sup_{s\in[0,t]}\big\|X^\varepsilon(s)-X^{\varepsilon'}(s)\big\|^2_{F^*_{1,2}}\Big]
+2\tilde{\alpha}\Bbb{E}\int_0^t\big|\Psi(X^\varepsilon(s))-\Psi(X^{\varepsilon'}(s))\big|_2^2ds\\
\leq\!\!\!\!\!\!\!\!&&\frac{1}{2}\Bbb{E}\Big[\sup_{s\in[0,t]}\big\|X^\varepsilon(s)-X^{\varepsilon'}(s)\big\|^2_{F^*_{1,2}}\Big]+
\tilde{\alpha}\Bbb{E}\int_0^t\big|\Psi(X^\varepsilon(s))-\Psi(X^{\varepsilon'}(s))\big|_2^2ds\\
&&+C_1\Bbb{E}\int_0^t\big\|X^\varepsilon(s)-X^{\varepsilon'}(s)\big\|^2_{F^*_{1,2}}ds
+C_2\Bbb{E}\int_0^t\big(\varepsilon|\Psi(X^\varepsilon(s))|^2_2+\varepsilon'|\Psi(X^{\varepsilon'}(s))|_2^2\big)ds.
\end{eqnarray*}
This yields
\begin{eqnarray}\label{eqn48}
&&\Bbb{E}\Big[\sup_{s\in[0,t]}\big\|X^\varepsilon(s)-X^{\varepsilon'}(s)\big\|^2_{F^*_{1,2}}\Big]
+2\tilde{\alpha}\Bbb{E}\int_0^t\big|\Psi(X^\varepsilon(s))-\Psi(X^{\varepsilon'}(s))\big|_2^2ds\nonumber\\
\leq\!\!\!\!\!\!\!\!&&C_1\Bbb{E}\int_0^t\big\|X^\varepsilon(s)-X^{\varepsilon'}(s)\big\|^2_{F^*_{1,2}}ds\nonumber\\
&&+C_2(\varepsilon+\varepsilon')\Bbb{E}\int_0^t\big(|\Psi(X^\varepsilon(s))|^2_2+|\Psi(X^{\varepsilon'}(s))|_2^2\big)ds.
\end{eqnarray}
Note that if the initial value $x\in F^*_{1,2}$ and \eqref{eq:2} is satisfied, we have \eqref{eqn43}. If $x\in L^2(\mu)$, we have \eqref{eqn7}, then \textbf{(H1)} implies that there exists a positive constant C such that
$$\sup_{\kappa\in(0,1)}\Bbb{E}\int_0^t|\Psi(X^\kappa(s))|^2_2ds\leq C.$$
Hence, by Gronwall's inequality and Young's inequality, we know that there exists a positive constant $C\in(0,\infty)$ which is independent of $\varepsilon, \varepsilon'$ such that
\begin{eqnarray}\label{eqn49}
&&\mathbb{E}\Big[\sup_{s\in[0,T]}\big\|X^\varepsilon(s)-X^{\varepsilon'}(s)\big\|^2_{F^*_{1,2}}\Big]+\mathbb{E}\int_0^T\big|\Psi(X^\varepsilon(s))-\Psi(X^{\varepsilon'}(s))\big|_2^2ds\nonumber\\
\leq\!\!\!\!\!\!\!\!&&C(\varepsilon+\varepsilon').
\end{eqnarray}
Hence, there exists an $\mathcal{F}_t$-adapted process $X\in L^2(\Omega;L^\infty([0,T];F^*_{1,2}))$ such that $X\in D([0,T];F^*_{1,2})$, $\Bbb{P}$-a.s., and $X^\varepsilon\rightarrow X$ in $L^2(\Omega;L^\infty([0,T]; F^*_{1,2}))$ as $\varepsilon\rightarrow0$. Furthermore, from Claim \ref{claim1}, we know that $X\in L^2([0,T]\times\Omega;L^2(\mu))$.

\vspace{2mm}
Using the similar argument as in Claim \ref{claim3}, we know that $X$ satisfies \eqref{eq:1} and $\int_0^\cdot\Psi(X(s))ds\in C([0,T];F_{1,2})$, $\Bbb{P}$-a.s.. This completes the existence proof for Theorem \ref{Th1}.

\vspace{2mm}
\textbf{Uniqueness}

Suppose $X_1$ and $X_2$ are two solutions to \eqref{eq:1}, we have $\Bbb{P}$-a.s.,
\begin{eqnarray}\label{eqn50}
&&\!\!\!\!\!\!\!\!X_1(t)-X_2(t)-L\int_0^t\Psi(X_1(s))-\Psi(X_2(s))ds\nonumber\\
=&&\!\!\!\!\!\!\!\!\int_0^t\int_{Z}\big(f(s,X_1(s-),z)-f(s,X_2(s-),z)\big)\widetilde{N}(ds,dz),\ \forall t\in [0,T].
\end{eqnarray}
Rewrite \eqref{eqn50} as following
\begin{eqnarray}\label{eqn51}
&&\!\!\!\!\!\!\!\!X_1(t)-X_2(t)+(1-L)\int_0^t\Psi(X_1(s))-\Psi(X_2(s))ds\nonumber\\
=&&\!\!\!\!\!\!\!\!\int_0^t\Psi(X_1(s))-\Psi(X_2(s))ds\nonumber\\
&&\!\!\!\!\!\!\!\!+\int_0^t\int_{Z}\big(f(s,X_1(s-),z)-f(s,X_2(s-),z)\big)\widetilde{N}(ds,dz),\ \forall t\in [0,T].
\end{eqnarray}

Apply It\^{o}'s formula to $\|X_1(t)-X_2(t)\|^2_{F^*_{1,2}}$ in $F^*_{1,2}$, we have
\begin{eqnarray}\label{eqn52}
&&\big\|X_1(t)-X_2(t)\big\|^2_{F^*_{1,2}}+2\int_0^t\big\langle\Psi(X_1(s))-\Psi(X_2(s)),X_1(s)-X_2(s)\big\rangle_2ds\nonumber\\
=\!\!\!\!\!\!\!\!\!&&2\int_0^t\big\langle
\Psi(X_1(s))-\Psi(X_2(s)),X_1(s)-X_2(s)\big\rangle_{F^*_{1,2}}ds\nonumber\\
&&+2\int_0^t\int_{Z}\big\langle
X_1(s-)-X_2(s-),f(s,X_1(s-),z)-f(s,X_2(s-),z)\big\rangle_{F^*_{1,2}}\widetilde{N}(ds,dz)\nonumber\\
&&+\int_0^t\int_{Z}\big\|f(s,X_1(s-),z)-f(s,X_2(s-),z)\big\|^2_{F^*_{1,2}}N(ds,dz).
\end{eqnarray}

Taking expectation to both sides of \eqref{eqn52}, \eqref{eqn46} and \textbf{(H3)} yield that
\begin{eqnarray*}
&&\Bbb{E}\big\|X_1(t)-X_2(t)\big\|^2_{F^*_{1,2}}+2\widetilde{\alpha}\Bbb{E}\int_0^t\big|\Psi(X_1(s))-\Psi(X_2(s))\big|_2^2ds\nonumber\\
\leq\!\!\!\!\!\!\!\!\!&&2\Bbb{E}\int_0^t\big\|\Psi(X_1(s))-\Psi(X_2(s))\big\|_{F^*_{1,2}}\cdot
\big\|X_1(s)-X_2(s)\big\|_{F^*_{1,2}}ds\nonumber\\
&&+C_2\Bbb{E}\int_0^t\big\|X_1(s)-X_2(s)\big\|^2_{F^*_{1,2}}ds.
\end{eqnarray*}
Using Young's inequality to the above inequality, and since $L^2(\mu)\subset
F^*_{1,2}$ continuously and densely, we obtain
\begin{eqnarray*}
&&\Bbb{E}\big\|X_1(t)-X_2(t)\big\|^2_{F^*_{1,2}}+2\widetilde{\alpha}\Bbb{E}\int_0^t\big|\Psi(X_1(s))-\Psi(X_2(s))\big|_2^2ds\nonumber\\
\leq\!\!\!\!\!\!\!\!\!&&2\widetilde{\alpha}\Bbb{E}\int_0^t\big|\Psi(X_1(s))-\Psi(X_2(s))\big|_2^2ds+
\frac{1}{2\widetilde{\alpha}}\Bbb{E}\int_0^t\big\|X_1(s)-X_2(s)\big\|^2_{F^*_{1,2}}ds\nonumber\\
&&+C_2\mathbb{E}\int_0^t\big\|X_1(s)-X_2(s)\big\|^2_{F^*_{1,2}}ds.
\end{eqnarray*}
Therefore,
\begin{eqnarray*}
&&\mathbb{E}\big\|X_1(t)-X_2(t)\big\|^2_{F^*_{1,2}}\leq
(\frac{1}{2\widetilde{\alpha}}+C_2)\mathbb{E}\int_0^t\big\|X_1(s)-X_2(s)\big\|^2_{F^*_{1,2}}ds.
\end{eqnarray*}
By Gronwall's lemma, we get $X_1(t)=X_2(t)$, $\Bbb{P}\text{-a.s.}$, $\forall t\in[0,T]$. Consequently, Theorem \ref{Th1} is completely proved. \hspace{\fill}$\Box$
\end{proof}

\end{document}